\documentclass[12pt]{amsart}
\usepackage{amsmath,amstext,amssymb,amsopn,amsthm}
\usepackage{mathtools}
\usepackage{enumerate}

\usepackage{color,graphicx}
\usepackage{mathrsfs,dsfont}

\numberwithin{equation}{section}

\allowdisplaybreaks

\usepackage[utf8]{inputenc}
\usepackage{hyperref}

\makeatletter
\g@addto@macro\th@plain{\thm@headpunct{}}
\makeatother

\numberwithin{equation}{section}

\newtheorem{thm}{Theorem}[section]
\newtheorem{lemma}[thm]{Lemma}
\newtheorem{Corollary}[thm]{Corollary}
\newtheorem{proposition}[thm]{Proposition}

\theoremstyle{definition}

\newtheorem{defin}[thm]{Definition}

\newtheorem{remark}[thm]{Remark}
\newtheorem{example}[thm]{Example}

\def \R {{\mathbb R}}
\def \eps {\varepsilon}

\newcommand{\dd}{\mathrm{d}}

\newcommand{\cM}{ {\mathcal M} }

\makeatletter
\@namedef{subjclassname@2020}{%
	\textup{2020} Mathematics Subject Classification}
\makeatother

\title[A phase transition for $\boxtimes$-convolution powers]{A phase transition for tails of the free multiplicative convolution powers}
\author[B. Ko\l{}odziejek]{Bartosz Ko\l{}odziejek}
\author[K. Szpojankowski]{Kamil Szpojankowski}
\email{b.kolodziejek@mini.pw.edu.pl}
\email{k.szpojankowski@mini.pw.edu.pl}
\address{
	Wydzia\l{} Matematyki i Nauk Informacyjnych\\
	Politechnika Warszawska\\
	ul. Koszykowa 75\\
	00-662 Warszawa, Poland.}


\begin{document}

\begin{abstract}
	We study the behavior of the tail of a measure $\mu^{\boxtimes t}$, where $\boxtimes t$ is the $t$-fold free multiplicative convolution power for $t\geq 1$. We focus on the case where $\mu$ is a probability measure on the positive half-line with a regularly varying right tail i.e. of the form $x^{-\alpha} L(x)$, where $L$ is slowly varying. We obtain a phase transition in the behavior of the right tail of $\mu^{\boxplus t}$ between regimes $\alpha<1$ and $\alpha>1$. Our main tool is a description of regular variation of the tail of $\mu$ in terms of the behavior of the corresponding $S$-transform at $0^-$. We also describe the tails of $\boxtimes$ infinitely divisible measures in terms of the tails of corresponding L\'evy measure, treat symmetric measures with regularly varying tails and prove the free analog of the Breiman lemma.
\end{abstract}
\maketitle
\section{Introduction}
	
Consider the space $\cM_+$ of probability measures on $[0,+\infty)$. For $\mu,\nu\in \cM_+$ we have natural operations of convolutions: additive denoted by $\mu\ast\nu$ and multiplicative $\mu\circledast\nu$, which correspond to distributions of the sum and the product of independent random variables. One of the important characteristics of a probability measure $\mu\in\cM_+$ is its tail behavior at $+\infty$, more precisely one can ask at what rate the function $\bar{\mu}(t) := \mu\big((t,+\infty)\big)$ goes to zero when $t\to+\infty$. Tail behavior determines many properties of a measure like existence of moments or its domain of attraction. One can also ask how tails of $\mu\ast\nu$ and $\mu\circledast\nu$ behave, given that the tail asymptotics of $\mu$ and $\nu$ are known. Clearly, since the multiplicative convolution corresponds to the distribution of $X\cdot Y$ for independent $X,Y$ one can reduce it to the case of additive convolution upon taking the logarithm (where one considers the operation $\ast$ on the set $\cM$ of all probability measures on $\R$). Tail asymptotics of additive convolutions is a classical and well studied problem. 

In parallel to the classical theory one can consider a non-commutative probability space, where classical random variables are replaced by non-commutative operators. In non-commutative probability there are several possible notions of independence, among which the most prominent is freeness defined by Voiculescu \cite{DVV86}. Freeness allows to define counterparts of operations $\ast$ and $\circledast$, the free additive convolution $\boxplus$ and the free multiplicative convolution $\boxtimes$.
In this paper we focus on free multiplicative convolutions. Similarly as $\circledast$, the operation $\boxtimes$ is a binary operation defined on $\cM_+\times \cM_+$, taking values in $\cM_+$. Given $\mu,\nu\in\cM_+$ one can consider two freely independent operators $X,Y$ with respective distributions $\mu$ and $\nu$. Then $X^{1/2}Y X^{1/2}$ has the distribution $\mu\boxtimes\nu$. One should note that, despite the fact that $X$ and $Y$ above do not commute, the operation $\boxtimes$ is actually commutative i.e. we always have $\mu\boxtimes\nu=\nu\boxtimes\mu$. In \cite{DVV87} Voiculescu showed existence of a transform (so-called $S$-transform) which is multiplicative under the operation $\boxtimes$ in $\cM_+$, that is one has:
\begin{align} \label{eq:Strans}
	S_{\mu\boxtimes\nu}(t)=S_{\mu}(t)S_{\nu}(t)\quad \mbox{for } t\in(-\varepsilon,0) 
\end{align}
for some $\varepsilon>0$. In \cite{DVV87}, $S$-transform is considered in a more general framework, where it takes complex argument. For our purposes it is enough to consider the $S$-transform as a function of real argument. Initially free multiplicative convolution was defined for compactly supported measures, in \cite{BV93} the authors extended the definition of free convolutions to all probability measures. Since equation \eqref{eq:Strans} is the only fact from free probability theory which is relevant for our paper, we decided to not introduce the notion of free independence, instead we refer to one of the monographs \cite{MS17,NS06}. 

The free multiplicative convolution appears naturally in the theory of random matrices. Assume that compactly supported probability measures $\mu$ and $\nu$ are the a.s. weak limits of eigenvalue distributions of  sequences of $N\times N$ matrices as $N\to+\infty$. Let us denote these sequences by $(X_N)_N$ and $(Y_N)_N$ and assume that distribution of one of these sequences is invariant by conjugation with independent Haar unitary matrices. Then, $\mu\boxtimes \nu$ is a.s. the weak limit of the eigenvalue distribution of product $X_N^{1/2} Y_N X_N^{1/2}$ provided $X_N$ and $Y_N$ are independent in the classical sense and $X_N$ is positive definite for each $N\in\mathbb{N}$ \cite{Vo91}.

In contrast to the commutative case, in the framework of non-commutative probability one cannot simply reduce the operation $\boxtimes$ to the additive case. In fact the free multiplicative convolution has some surprising properties. Such differences are manifested for example in the context of infinitely divisible distributions. Denote by $ID(\ast)$ and $ID(\boxplus)$ the set of infinitely divisible measures under $\ast$ and $\boxplus$ respectively. It was observed in \cite{BPB} that there is a bijection $\mathbb{B}\colon ID(\ast)\to ID(\boxplus)$, such that for each $\mu\in ID(\ast)$ its domain of attraction is the same as the domain of attraction of $\mathbb{B}(\mu)\in ID(\boxplus)$. One can consider similar mapping between the sets $ID(\circledast)$ and $ID(\boxtimes)$. In \cite{BP00} the authors showed that such mapping can be defined, but fails to work as well as in the additive case. Another difference can be observed in the context of laws of large numbers. For $c>0$ we denote by $D_c(\mu)$ the pushforward of $\mu$ by the mapping $x\mapsto cx$, that is $D_c(\mu)(A)=\mu(A/c)$ for any Borel set $A$. Classical law of large numbers says that $D_{1/n}(\mu^{\ast n})\to\delta_\alpha$ weakly, where $\alpha$ is the mean value of $\mu$. A similar results holds for the free additive convolution, that is we have $D_{1/n}(\mu^{\boxplus n})\to\delta_\alpha$. The law of large numbers for classical multiplicative convolution follows again from the additive case, more precisely let $\psi_c(\mu)$ be the pushforward of $\mu$ corresponding to the mapping $x\mapsto x^{c}$ for $x,c>0$, then one immediately has $\psi_{1/n}(\mu^{\circledast n })\to \delta_\beta$ with $\beta=\exp\left(\int\log(t)\mu(\dd t)\right)$. For the free multiplicative case it was shown in \cite{HaM13} that $\psi_{1/n}(\mu^{\boxtimes n })\to\nu$ weakly, where in general $\nu$ is not a Dirac delta measure.

The surprising behavior of the free multiplicative convolution motivates further study of this operation, in particular the study of the related tail behavior. Tail behavior of the free additive convolutions was studied in \cite{HM13} in the context of so-called free subexponentiality. In particular the authors showed that if $\mu$ is such that its tail $\bar{\mu}$ is regularly varying then $\bar{\mu}^{\boxplus n}(t)\sim n \bar{\mu}(t)$, where $f(t) \sim g(t)$ means that the quotient $f(t)/g(t)$ goes to $1$ as $t$ tends to $+\infty$. Let us stress that such behavior coincides with the classical case of $\ast$. A related studies were performed for Boolean additive and multiplicative convolutions in \cite{CH18}.  

We also work in the framework of measures with regularly varying tails. Let us briefly recall some basic facts about regularly varying functions. A function $f$ is called regularly varying with index $\alpha$ if for every $\lambda>0$ one has $f(\lambda x )/f(x)\to \lambda^\alpha$ as $x\to+\infty$, when $\alpha=0$ then $f$ is called slowly varying. For $\rho\in \R$ we write $\mathcal{R}_\rho$ for the class of regularly varying functions with index $\rho$, which
consists of functions $f$ of the form $f(x) = x^\rho L(x)$ for some $L\in \mathcal{R}_0$.

\begin{defin}
	We say that $\mu$ has regularly varying tail at infinity with index $-\alpha\leq0$ if
	\begin{align}\label{Introeq:L}
		\bar{\mu}(x)=\mu\big((x,+\infty)\big)\sim \frac{L(x)}{x^\alpha}
	\end{align}
	for a slowly varying function $L$.
\end{defin}

Our main result is a free probability analogue of the following classical result. Point $(i)$ below follows from additive subexponentiality of the pushforward measure of $\mu$ under the $\log$, while $(ii)$ can be found e.g. in \cite{JM06}.
\begin{thm}\label{Introthm:classConv}\  
	\begin{enumerate}[(i)]
		\item Assume that $\mu\in\mathcal{M}_+$, $L\in\mathcal{R}_0$ and $\beta>0$. If $\bar{\mu}(x)\sim L(\log(x))/\log(x)^\beta$, then
		\[
		\mu^{\circledast n}\big((x,+\infty)\big)\sim n\, \bar{\mu}(x).
		\]
		\item 	If $\bar{\mu}(x)\sim c \,x^{-\alpha}$ with $c, \alpha>0$, then 
		\[
		\mu^{\circledast n}\big((x,+\infty)\big)\sim \frac{\alpha^{n-1} c^{n-1}}{(n-1)!} \log^{n-1}(x)\bar{\mu}(x).
		\] 
	\end{enumerate}
\end{thm}

The behavior of the tails of the free multiplicative convolution is much more complicated. The result below is another evidence that the free multiplicative convolution behaves in a distinctive way. We observe a phase transition for \eqref{Introeq:L} between the cases $0<\alpha<1$ and $\alpha>1$. 
We observe that the case $\alpha=1$ with first moment finite is the same as $\alpha>1$. Moreover we describe the case $\alpha=1$ when the slowly varying function is constant, surprisingly this case is very similar to the classical case. Also in the case when the tail is slowly varying, corresponding to $\alpha=0$, we work with the special case $\bar{\mu}(x)\sim L(\log(x))/\log(x)^\beta$ similarly as in the theorem above.
Recall that for the free multiplicative convolution one can consider fractional convolution powers $\mu^{\boxtimes t}$ for any $t\geq 1$, see \cite{BB05}. 
\newpage\begin{thm}\label{Introthm:1.4}\  
	\begin{enumerate}[(i)]
		\item If $\mu\in\mathcal{M}_+$ and $\bar{\mu}(x)\sim L(\log(x))/\log(x)^\beta$ for $L\in\mathcal{R}_0$ and $\beta>0$, then for $t\geq 1$,
		\[
		\mu^{\boxtimes t}\big((x,+\infty)\big)\sim t^\beta\,\mu\big((x,+\infty)\big).
		\]
		\item  Assume $\mu\in\mathcal{M}_+$  satisfies \eqref{Introeq:L} for $\alpha\in(0,1)$. Then for $t\geq 1$ we have $\mu^{\boxtimes t}\big((\cdot,+\infty)\big)\in \mathcal{R}_{-\alpha_t}$, where 
		\[
		\alpha_t=\frac{\alpha}{\alpha+t(1-\alpha)}.
		\]
		In particular, if $\bar{\mu}(x)\sim c/x^{\alpha}$ for some $\alpha\in(0,1)$ and $c>0$, then for $t\geq 1$ one has,
		\begin{align*}
			\mu^{\boxtimes t}\big((x,+\infty)\big)\sim \frac{c_{t,\alpha}}{x^{\alpha_t}},
		\end{align*}
		where 
		\[
		c_{t,\alpha} = \left(c\frac{\pi\alpha}{\sin(\pi\alpha)}\right)^{t/(\alpha+t(1-\alpha))} \frac{\sin(\pi\alpha_t)}{\pi\alpha_t}.
		\]
		\item If $\mu\in\mathcal{M}_+$ is such that $\bar{\mu}(x)\sim c/x$ with $c>0$, then for $t\geq 1$,
		\begin{align*}
			\mu^{\boxtimes t}\big((x,+\infty)\big)\sim c^{t-1}t \log(x)^{t-1}\mu\big((x,+\infty)\big).
		\end{align*}
\item	Let $\alpha\geq1$ and assume $\mu\in\mathcal{M}_+$  satisfies \eqref{Introeq:L} and has first moment (denoted $m_1(\mu)$) finite.   Then for $t\geq 1$ we have 
		\begin{align*}
			\mu^{\boxtimes t}\big((x,+\infty)\big)\sim t\, m_1(\mu)^{\alpha(t-1)} \mu\big((x,+\infty)\big).
		\end{align*}
	\end{enumerate}
\end{thm}

In order to illustrate the phase transition phenomenon let us examine the tail behavior of the free convolution powers for the family of Pareto distributions.

\begin{example}[Phase transition for Pareto distributions]\label{intro:example}
	Consider the family $\left(\mu_\alpha\right)_{\alpha>0}$ of Pareto distribution given by the density $\mu_\alpha(\dd x)=\alpha/x^{\alpha+1} \mathit{1}_{[1,+\infty)}(x)\dd x$. Then for $x>1$ we have $\mu_\alpha\big((x,+\infty)\big)=1/x^\alpha$, also for $\alpha>1$ we get $m_1(\mu_\alpha)=\alpha/(\alpha-1)$. Then for the tail of $\mu_\alpha^{\boxtimes t}$, with notations as in the theorem above, we get
	\begin{align*}
		\mu_\alpha^{\boxtimes t}\big((x,+\infty)\big)\sim
		\begin{cases}
			c_{t,\alpha} \,x^{-\tfrac{\alpha}{\alpha+t(1-\alpha)}}&\mbox{ for }\alpha\in(0,1), 
			\\
			t\log(x)^{t-1} x^{-1}&\mbox{ for }\alpha=1,
			\\
			t \left(\frac{\alpha}{\alpha-1}\right)^{\alpha (t-1)}x^{-\alpha}&\mbox{ for }\alpha>1.
		\end{cases} 
	\end{align*}
\end{example}

Our main tool which allows us to prove the theorem above is a detailed description of the  asymptotic properties of $S$-transforms, which allows us to completely characterize the behavior of the $S$-transform of measures with regularly varying tails. Depending on the index, we observe several different regimes. 
We present the results about the relation between the $S$-transform and the tail at $+\infty$ in Section \ref{sec:4}.

The tools developed in Section \ref{sec:4} allows us also to study $\boxtimes$ infinitely divisible probability measures. They are parameterized in terms of a real number $\gamma$ and a finite measure $\sigma$ on $[0,+\infty]$, thus we will write $\mu_\boxtimes^{\gamma,\sigma}$, for a $\boxtimes$ infinitely divisible measure parameterized by $\gamma$ and $\sigma$. We find a precise relation between the left (resp. right) tail of $\sigma$ and right (resp. left) tail of $\mu_\boxtimes^{\gamma,\sigma}$. The precise statement is more complicated than for the case of convolution powers, we refer to Section \ref{sec:6}. Let us just mention that again we obtain a behavior which resembles Example \ref{intro:example}. We observe a phase transition for the tail of $\mu_\boxtimes^{\gamma,\sigma}$ with critical case when $\sigma$ has regularly varying left tail with parameter $\alpha=1$. We refer to Example \ref{example:1.66} for an explicit example of a family of measures $\sigma$ for which we get a discontinuous change of tail asymptotics of corresponding measures $\mu_\boxtimes^{\gamma,\sigma}$.

This paper has 6 more sections. In Section \ref{sec:prelim} we introduce all necessary facts from the theory of regularly varying functions and some Tauberian theorems. In Section \ref{sec:3} we introduce some facts from so-called free harmonic analysis, we define all transforms that we need, present some of their basic properties and we make some useful new observations about the $S$-transform. In Section \ref{sec:4} we present our main technical results, that is, for $\mu\in\cM_+$ with regularly varying tail we characterize 
 the behavior of the $S$-transform of $\mu$ at $0^-$. In Section \ref{sec:5} we prove Theorem \ref{Introthm:1.4}. Section \ref{sec:6} is devoted to the study of tails $\boxtimes$-infinitely divisible measures. In Section \ref{sec:7} we gather some more consequences of the results of Section \ref{sec:4} which are not directly related with the phase transition phenomenon. In particular we discuss free counterpart of Breiman's lemma, we determine the left tail of $\mu$ in terms of its $S$-transform and discuss how our results extends to symmetric probability measures.

\section{Preliminaries}\label{sec:prelim}
In this section we discuss in detail the theory of regular variation, Tauberian theorems that we need and also some aspects of free harmonic analysis which we use in the subsequent sections. We intend to introduce the reader with all background needed to understand our results.
\subsection{Regular variation}

We say that two functions $f, g$ defined on a neighborhood of infinity are asymptotically equivalent if $f(x)\sim g(x)$.
Recall that $\mathcal{R}_\rho$ denotes the class of regularly varying functions with index $\rho\in\R$. Function $L\in\mathcal{R}_0$ is called slowly varying and we have $L(x)=o(x^\varepsilon)$ as $x\to+\infty$ for any $\varepsilon>0$.
\begin{thm}[\mbox{\cite[Section 1.7.7]{BGT89}}]\	\label{thm:conj}\ \\
If $f\in \mathcal{R}_\alpha$ with $\alpha\neq 0$, then there exists $g\in \mathcal{R}_{1/\alpha}$ such that
	\[
	f(g(x))\sim g(f(x))\sim x.
	\]
	Function $g$ is determined uniquely up to asymptotic equivalence. Moreover if $f(x)\sim x^{ab}L^{a}(x^b)$ for $L\in\mathcal{R}_0$ and $a,b>0$, then 
	\[
	g(x)\sim x^{1/(ab)}{L^\#}^{1/b}(x^{1/a}),
	\]
	where 
	$L^\#$ is a slowly varying function, unique up to asymptotic equivalence, with
	\[
	\lim_{x\to+\infty}L(x) L^\#(x L(x))= 1\qquad\mbox{and}\qquad 	\lim_{x\to+\infty}L^\#(x) L(x L^\#(x))=1.
	\]
	In particular, if
	$L(x) = \prod_{k=1}^n (\log^{(\circ k)}(x))^{\alpha_k}$, where $\log^{(\circ k)}$ is the $k$ fold composition of $\log$, then
	\[
	L^\#(x)\sim 1/L(x).
	\]
\end{thm}
For a fixed function $L$, the function $L^\#$ goes under the name de Bruijn conjugate of $L$.

We will also use the ``second-order theory'' of regular variation due to de Haan. First we need to introduce a proper subclass of $\mathcal{R}_0$.
\begin{defin}
	If $L\in \mathcal{R}_0$, we write $f\in\Pi_L(c)$ with $c\in\R$ if 
	\[
	\lim_{x\to+\infty}\frac{f(\lambda x)-f(x)}{L(x)}=c \log(\lambda)
	\]
	for all $\lambda>0$.
\end{defin}

\begin{remark}\label{rem:PiInv}
If $f\in \Pi_L(c)$ with $c>0$, then $f(x)/L(x)\to+\infty$ as $x\to+\infty$, \cite[Proposition 1.5.9a]{BGT89}. Moreover, $1/f\in \Pi_{L/f^2}(-c)$, \cite[Proposition 1]{HR79}.
\end{remark}

Observe that $f(\lambda x)/f(x)=1+\left[(f(\lambda x)-f(x))/L(x)\right]\cdot \left[L(x)/f(x)\right]$, hence by the above remark and the definition of $\Pi_L(c)$,
for any $c>0$ and $L\in\mathcal{R}_0$, we have that $\Pi_L(c)\subset \mathcal{R}_0$.

\begin{thm}[\mbox{\cite[Theorem 3.6.6, Theorem 3.7.1]{BGT89}}] \label{thm:PiL}
	The following are equivalent: 
	\begin{enumerate}[(i)]
		\item $f\in\Pi_L(c)$,
		\item There exist constants $x_0\geq0$ and $C,d\in\R$ such that for $x\geq x_0$,
		\[
		f(x)=C+c\int_{x_0}^x(1+o(1))\frac{L(t)}{t} \dd t+d(1+o(1))L(x),
		\]
		where both functions $o(1)$ are measurable,
		\item For any $\sigma>0$,
		\[
		\sigma x^\sigma \int_x^\infty u^{-\sigma-1}f(u)\dd u-f(x)\sim \frac{c}{\sigma}L(x).
		\]
	\end{enumerate}
\end{thm}

Assume that $L\in \mathcal{R}_0$ and $f_1\in\Pi_L(c)$ for some $c\neq0$. We say that $f_1$ and $f_2$ are $\Pi_L$-equivalent (denoted $f_1(x)\sim_{\Pi_L} f_2(x)$) if the limit
\[
\lim_{x\to+\infty}\frac{f_1(x)-f_2(x)}{L(x)}
\]
exists and is finite.
Clearly, $\sim_{\Pi_L}$ is an equivalence relation on $\cup_{c}\Pi_L(c)$. 

As follows from Theorem \ref{thm:conj}, de Bruijn conjugate pairs $(L, L^\#)$  appear naturally when we look for asymptotic inverses of regularly varying functions. However, when dealing with slowly varying functions in the subclass $\Pi_L$, a stronger notion of conjugacy is required.

Henceforth we will use notation $f^{\langle -1 \rangle}$ for the compositional inverse of $f$.
\begin{defin}\label{def:conj}
Assume that function $f\in \Pi_L(c)$ is continuous and such that $x\mapsto xf(x)$ is increasing. Without loss of generality we may assume that $c\in\{-1,1\}$. Define a function $v$ on $(0,+\infty)$ by
\[
v(x)=x f(x),
\]
which clearly is invertible. Writing its inverse as, 
\[
v^{\langle-1\rangle}(x)=x f^\ast(x),
\]
the function $f^\ast\colon (0,+\infty)\to(0,+\infty)$ is called the $\Pi$-conjugate function for $f\in \Pi_L(\pm1)$. Equivalently, $f^\ast$ can be defined through the condition, \cite[Page 1032]{HR79},
\[
\lim_{x\to+\infty}\frac{f(x)}{L(x)}\big(f(x) f^\ast(x f(x))-1\big)=0.
\]
\end{defin}

Observe that since $f(x)/L(x)\to+\infty$ as $x\to+\infty$, function $f^\ast$ is also de Bruijn conjugate for $f$.
\begin{thm}[\mbox{\cite[Theorem 1]{HR79}}]\label{thm:PiConj}\ 
	\begin{enumerate}[(i)]
		\item $f\in\Pi_L(\pm1)$ if and only if
		\[
		f^\ast\in \Pi_M(\mp1),
		\]
		where $M(x):=L(x f^\ast(x))f^\ast(x)^2\in \mathcal{R}_0$.
		\item $f^\ast$ is unique up to the equivalence class $\sim_{\Pi_M}$  on $\Pi_M$.
		\item $f^{\ast\ast}(x)\sim_{\Pi_L} f(x)$.
	\end{enumerate}
\end{thm}

We say that a measurable function $f$ is rapidly varying (denoted $f\in \mathcal{R}_\infty$) if 
\begin{align*}
\lim_{x\to+\infty}\frac{f(\lambda x)}{f(x)}=+\infty
\end{align*}
for all $\lambda>1$. We are interested in a proper subclass $K\mathcal{R}_\infty$ of $\mathcal{R}_\infty$, which is classically defined as a family of measurable functions for which Karamata indices are $+\infty$  \cite[Section 2.4]{BGT89}. Equivalently, one can define it as follows (see \cite[Proposition 2.4.3]{BGT89}).
\begin{defin}\label{def:KR}
We write $f\in K\mathcal{R}_\infty$ if and only if $f$ is measurable and for every $d\in\R$,
\[
\liminf_{x\to+\infty} \inf_{\lambda\geq1} \frac{f(\lambda x)}{\lambda^d f(x)}\geq1.
\]	
\end{defin}

Below we list  important properties of $K\mathcal{R}_\infty$. 
\begin{thm}\label{thm:KR}\ 
	\begin{enumerate}[(i)]
		\item \cite[Theorem 2.4.7]{BGT89} Assume that $f$ is continuous and increasing to $+\infty$. Then,
\[
f\in \mathcal{R}_0\qquad\mbox{if and only if}\qquad f^{\langle-1 \rangle}\in K\mathcal{R}_\infty.
\]
		\item \cite[Proposition 2.4.4]{BGT89} If $f\in \mathcal{R}_\infty$ is nondecreasing, then $f\in K\mathcal{R}_\infty$.
	   \item \cite[Theorem 2.4.5]{BGT89}  $f\in K\mathcal{R}_\infty$ if and only if
	   \[
	   f(x)=\exp\left(z(x)+\eta(x)+\int_1^x \xi(t)\frac{\dd t}{t}\right),\qquad x\geq1,
	   \]
	   where functions $z$, $\eta$ and $\xi$ are measurable and such that $z$ is nondecreasing, $\eta(x)\to0$ and $\xi(x)\to+\infty$ as $x\to+\infty$.
	\end{enumerate}
\end{thm}

We will need also the following easy result.
\begin{remark}\label{rem:KR}
	Let $f$ be measurable and increasing. Then,
	\[
	f(x)\in K\mathcal{R}_\infty\qquad\mbox{if and only if}\qquad (x-1)f(x)\in K\mathcal{R}_\infty.
	\]
	Indeed, it follows from Definition \ref{def:KR} and the fact that for $x\geq2$ and $\lambda\geq1$  we have
	\[
\frac{f(\lambda x)}{\lambda^d f(x)}\leq 	\frac{f(\lambda x)(\lambda x-1)}{\lambda^d f(x)(x-1)}\leq \frac{f(\lambda x)}{\lambda^{d-2} f(x)}.
	\]
\end{remark}

\subsection{Transforms}

Let $\mathcal{M}_+$ denote the set of Borel probability measures on $\R_+=[0,+\infty)$.
For $\mu\in\mathcal{M}_+$ and $p\in\R$ denote
\[
m_p(\mu):=\int_{[0,+\infty)} t^p\mu(\dd t).
\]
\begin{defin}
	For $p\in\mathbb{N}\cup\{0\}$ denote by $\mathcal{M}_p$ the set on measures which have finite moments only up to $p$th one, i.e. 
	\[
	\mathcal{M}_p=\{ \mu\in\mathcal{M}_+\colon m_p(\mu)<+\infty\mbox{ and }m_{p+1}(\mu)=+\infty  \}.
	\]
\end{defin}

In what follows for $\mu\in\mathcal{M}_+$ we will denote its tail by
\[
\bar{\mu}(t) := \mu\big((t,+\infty)\big),\qquad t\geq 0.
\]
Let us make some simple observation about regularly varying tails following from the fact that $m_p(\mu)=p\int_0^{+\infty} t^{p-1}\bar{\mu}(t)\,\dd t$, $p\in\mathbb{N}$.
\begin{remark}\label{rem:ap}
	Assume that $\mu\in\mathcal{M}_+$ is such that $t\mapsto\bar{\mu}(t)$ is regularly varying with index $-\alpha\leq0$ and slowly varying function $L$. 
	
	Then, $\mu$ belongs to $\mathcal{M}_p$, $p\in\mathbb{N}$, if and only if one of the following conditions is satisfied
	\begin{enumerate}[(a)]
		\item $\alpha\in(p,p+1)$,
		\item $\alpha=p$ and $\int_1^{+\infty} L(t)/t\,\dd t<+\infty$,
		\item $\alpha=p+1$ and $\int_1^{+\infty} L(t)/t\,\dd t=+\infty$.
	\end{enumerate}
Probability measure $\mu$ belongs to $\mathcal{M}_0$ if and only if one of the following conditions is satisfied 
	\begin{enumerate}[(a$^\prime$)]
	\item $\alpha\in[0,1)$,
	\item $\alpha=1$ and $\int_1^{+\infty} L(t)/t\,\dd t=+\infty$.
\end{enumerate}
\end{remark}

We conclude this section with a result similar to \cite[Lemma 1]{BG06}. We will refer to the type of expansion appearing in \eqref{eq:oneS} as ``one sided Taylor expansion''.
\begin{lemma}\label{lem:expansion}
	Consider a function $f\colon (-\delta,0)\to \mathbb{R}$ and assume that $f$ is $C^\infty(-\delta,0)$ for some $\delta>0$. Fix an integer $p$ and suppose that there exists $\varepsilon\in[0,1)$ such that 
	\begin{align}\label{eq:succ}
	f^{(p)}\left(z\right)= o(z^{-\varepsilon}),\qquad z\to 0^-.
	\end{align}
	Then, there exists a real sequence $(f_k)_{k=0}^{p-1}$ such that, for each $n\in\{0,\ldots,p-1\}$ we have
	\begin{align}\label{eq:oneS}
	f^{(n)}(z)=\sum_{k=0}^{p-1-n}  \frac{f_{k+n}}{k!}z^k + o(z^{p-n-\varepsilon}),\qquad z\to 0^-.
	\end{align}
\end{lemma}
The above result is proved by the successive integration of \eqref{eq:succ} over interval $(z,0)$, $z<0$. Then, we clearly have $f_n=f^{(n)}(0^-)$.

\subsection{Tauberian theorems}
This section is devoted to Tauberian theorems. We state some known result that we need in the subsequent sections, we present also proofs of results which are not available in research literature, thus some results of this subsection are new. 

\begin{thm}[\mbox{\cite[Theorem 1.6.4]{BGT89}}] \label{thm:TaubMom}
	Assume that $\mu\in\mathcal{M}_+$. For any number $\alpha$ and $n$ such that $0<\alpha<n$, the following assertions are equivalent:
	\begin{align*}
		\bar{\mu}(x)&\sim x^{-\alpha}L(x), \\
		\int_{[0,x]} t^n \mu(\dd t)&\sim \frac{\alpha}{n-\alpha} x^{n-\alpha} L(x).
	\end{align*}
\end{thm}

\begin{thm}[\mbox{\cite[Theorem 1.7.4]{BGT89}}] \label{thm:TaubSt}
	Assume that $U$ is nondecreasing, $U(0^-)=0$ and $\rho>0$. Then, if $c$ is a positive constant, $0\leq\sigma<\rho$ and $L\in \mathcal{R}_0$, the following assertions are equivalent:
	\begin{align*}
		U(x)&\sim c\,x^{\sigma}L(x), \\
		\int_{[0,+\infty)} \frac{\dd U(t)}{(t+x)^\rho}&\sim c\, \frac{\Gamma(\rho-\sigma)\Gamma(\sigma+1)}{\Gamma(\rho)}x^{\sigma-\rho}L(x).
	\end{align*}
\end{thm}

\begin{Corollary}\label{cor:taub0}
	Let $L\in\mathcal{R}_0$ and let  $\alpha\in[p,p+1)$ for some $p\in\mathbb{N}\cup\{0\}$.
	The following assertions are equivalent:
	\begin{enumerate}[(i)]
		\item $\bar{\mu}(x)\sim x^{-\alpha}L(x)$, 
		\item 
		$\int_{[0,x]} t^{p+1}\mu(\dd t) \sim 
		\frac{\alpha}{p+1-\alpha} x^{p+1-\alpha}L(x)$,
		\item $\int_{[0,+\infty)} \frac{ t^{p+1}}{(t+x)^{p+2}}\mu(\dd t) \sim \frac{\alpha \Gamma(\alpha+1)\Gamma(p+1-\alpha)}{(p+1)!} x^{-\alpha-1}L(x)$.
	\end{enumerate}
\end{Corollary}
\begin{proof}
	Equivalence between $(i)$ and $(ii)$ follows from Theorem \ref{thm:TaubMom}.
	
	To prove equivalence between $(ii)$ and $(iii)$ define $U(x):=\int_{[0,x]} t^{p+1}\mu(\dd t)$. In such case,
	\[
	\int_{[0,+\infty)} \frac{ t^{p+1}}{(t+x)^{p+2}}\mu(\dd t) = \int_{[0,+\infty)} \frac{ \dd U(t)}{(t+x)^{p+2}}.
	\]
	Moreover, $U$ is nondecreasing and $U(0^-)=0$. Thus, the result follows from Theorem \ref{thm:TaubSt} with $\rho=p+2$, $c=\alpha/(p+1-\alpha)$ and $\sigma=p+1-\alpha$.
\end{proof}

\begin{defin}
	For a kernel $k\colon(0,+\infty)\to\R$ its Mellin transform is defined as
	\[
	\check{k}(z):=\int_0^\infty \frac1{t^{z}} k(t)\frac{\dd t}{t}
	\]
	for $z\in\mathbb{C}$ such that the integral converges. 
	
	For functions $k, f\colon(0,+\infty)\to\R$ we define their Mellin convolution by
	\[
	k\stackrel{M}{\ast}f(x):=\int_0^\infty k\left(\frac xt\right) f(t)\frac{\dd t}{t}
	\]
	for those $x>0$ for which the integral converges.
\end{defin}

\begin{defin}
	For $p\in\mathbb{N}\cup\{0\}$ we define kernel $k_p\colon \R\to [0,1]$ by
	\begin{align}\label{def:kp}
		k_p(x) = x^{p+1}I_{[0,1]}(x).
	\end{align}
\end{defin}
	It can be easily checked that $\check{k}_p$ is absolutely convergent for $\Re z<p+1$ and for such $z$ we have
\begin{align}\label{eq:kpcheck}
	\check{k}_p(z) = \frac{1}{1+p-z}.
\end{align}	

\begin{thm}\label{thm:taub0}
	Let $\gamma<1+p$ and let $L\in\mathcal{R}_0$. Assume that 
	\begin{align}\label{eq:1}
		f(x)\sim x^\gamma L(x).
	\end{align}
	Then,
	\begin{align}\label{eq:2}
		k_p\stackrel{M}{\ast}f(x)\sim \check{k}_p(\gamma)x^\gamma L(x).
	\end{align}
	If $f$ is monotonic, then \eqref{eq:2} implies \eqref{eq:1}.
\end{thm}
\begin{proof}
	Going from \eqref{eq:1} to \eqref{eq:2} is standard and follows from \cite[Theorem 4.1.6]{BGT89}. 
	
	Assume \eqref{eq:2}. Without loss of generality, we assume that $f$ is nondecreasing (if not, consider $-f$).  Let 
	\[
	F(x) := x^{-p-1}\left(k_p\stackrel{M}{\ast}f(x)\right) = \int_x^\infty \frac{f(t)}{t^{p+2}}\dd t.
	\] 
	Recall that $\check{k}_p(\gamma)=\frac{1}{1+p-\gamma}$.
	We have
	\[
	F(x)\sim \frac{1}{1+p-\gamma} x^{\gamma-p-1}L(x).
	\]
	By monotonicity of $f$, for any $\beta>1$ and $x>0$, we have
	\[
	F(x)-F(\beta x) = \int_x^{\beta x} \frac{f(t)}{t^{p+2}}\dd t \geq \frac{1}{p+1}\frac{f(x)}{x^{p+1}}\left( 1 - \frac1{\beta^{p+1}} \right).
	\]
	Thus,
	\[
	\frac{f(x)}{x^\gamma L(x)}\leq (p+1)\frac{F(x)-F(\beta x)}{x^{\gamma-p-1} L(x)}\frac{\beta^{p+1}}{\beta^{p+1}-1}.
	\]
	Whence,
	\[
	\limsup_{x\to+\infty} \frac{f(x)}{x^\gamma L(x)}\leq (p+1) \frac{1-\beta^{\gamma-p-1}}{1+p-\gamma}\frac{\beta^{p+1}}{\beta^{p+1}-1} 
	\]
	and the r.h.s. above converges to $1$ as $\beta\to1^+$. 
	
	Let now $\beta<1$. Again by monotonicity we have
	\[
	F(\beta x)-F(x) \leq \int_{\beta x}^x \frac{f(t)}{t^{p+2}}\dd t\leq \frac{f(x)}{x^{p+1}} \frac{1}{p+1}\left(\frac{1}{\beta^{p+1}}-1\right).
	\]
	We proceed similarly as before and show that 
	\[
	\liminf_{x\to+\infty} \frac{f(x)}{x^\gamma L(x)} \geq 1.
	\]
\end{proof}

\begin{thm}\label{thm:tauber}
	Assume that $k$ is a nonnegative kernel such that $\check{k}(z)$ is absolutely convergent in a strip $|\Re z|\leq \eta$ for some $\eta>0$ and $\check{k}(z)\neq 0$ for $\Re z=0$. Let $f\colon(0,+\infty)\to\R$ be measurable, $L\in \mathcal{R}_0$ and $c\neq0$.
	\begin{enumerate}[(i)]
		\item Assume additionally that $x^\eta f(x)$ is bounded on every interval $(0,a]$.  Then
		\begin{align}\label{eq:abel20}
			f (x)\sim c L(x)
		\end{align}
		implies
		\begin{align}\label{eq:taub20}
			k\stackrel{M}{\ast}f(x)\sim \check{k}(0) c L(x).
		\end{align}
		Conversely, if $f$ is monotonic, then \eqref{eq:taub20} implies \eqref{eq:abel20}.	
		\item Assume additionally that $f$ is locally bounded on $[0,+\infty)$. Then
		\begin{align}\label{eq:abel2}
			f \in \Pi_L(c)
		\end{align}
		implies
		\begin{align}\label{eq:taub2}
			k\stackrel{M}{\ast}f\in \Pi_L(\check{k}(0) c).
		\end{align}
		Conversely, if $f$ is monotonic, then \eqref{eq:taub2} implies \eqref{eq:abel2}.	
	\end{enumerate}
\end{thm}

\begin{remark}
	The above theorem follows from Section \ref{sec:4} in \cite{BGT89}, however is not stated as a single theorem. Let us explain briefly which results from \cite{BGT89} we use here. 
	
	The implication from \eqref{eq:abel20} to \eqref{eq:taub20} follows directly from \cite[Theorem 4.1.6]{BGT89} with $\rho=0$ and  \eqref{eq:abel2} implies \eqref{eq:taub2} by \cite[Section 4.11.1]{BGT89}. 
	The converse implications holds by  \cite[Theorem 4.8.3]{BGT89} and \cite[Theorem 4.11.2]{BGT89}, respectively, under a Tauberian condition
	\begin{align}\label{eq:taubcond2}
		\lim_{\lambda\to1^+} \limsup_{x\to+\infty} \sup_{\mu\in [1,\lambda]} \frac{f(x) - f(\mu x)}{L(x)}= 0,
	\end{align}
	which is clearly satisfied if $f$ is nondecreasing. If $f$ is nonincreasing, then we consider $-f$.
\end{remark}

In the next theorem we state some useful conditions for a measure to have a regularly varying tails.

\begin{thm}\label{thm:taubPi}
	Let $L\in\mathcal{R}_0$ and $p\geq 0$.
	The following assertions are equivalent:
	\begin{enumerate}[(i)]
		\item $\bar{\mu}(x)\sim x^{-p-1} L(x)$,
		\item 
		$x\mapsto \int_{[0,x]} t^{p+1}\mu(\dd t) \in \Pi_L(p+1)$,
		\item $x\mapsto x^{p+2}\int_{[0,+\infty)} \frac{ t^{p+1}}{(t+x)^{p+2}}\mu(\dd t) \in \Pi_L(p+1)$.
	\end{enumerate}
\end{thm}
\begin{proof}
	First we show the equivalence between $(i)$ and $(ii)$. Let us denote $\sigma=p+1$.
	Observe that by Tonelli's theorem we have
	\begin{align*}
		\int_{[0,x]} t^\sigma\mu(\dd t) &= \int_{[0,x]} \int_0^t \sigma s^{\sigma-1}\dd s\,\mu(\dd t) = \sigma \int_0^x s^{\sigma-1} \mu\big((s,x]\big) \dd s \\
		&=   \sigma \int_0^x s^{\sigma-1} \bar{\mu}(s) \dd s - x^{\sigma}\bar{\mu}(x).
	\end{align*}
	Thus, if $(i)$ is assumed, then
	\[
	\int_{[0,x]} t^\sigma\mu(\dd t) =  \sigma\int_0^x (1+o(1))\frac{L(s)}{s} \dd s - L(x)(1+o(1)).
	\]
	By Theorem \ref{thm:PiL} we obtain the result.
	
	Conversely, assume $(ii)$ and observe that with $f(x):=\int_{[0,x]} t^\sigma\mu(\dd t)$, by $(iii)$ of Theorem \ref{thm:PiL} we obtain that 
	\begin{align*}
		\sigma x^\sigma \int_x^\infty u^{-\sigma-1} f(u)\dd u-f(x) \sim L(x).
	\end{align*}
	But the left hand side is equal to
	\begin{align*} 
		\sigma x^\sigma \int_x^\infty u^{-\sigma-1} \left(f(u)-f(x) \right)\dd u &= \sigma x^\sigma \int_x^\infty u^{-\sigma-1} \int_{(x,u]} t^\sigma \mu(\dd t)\,\dd u \\
		&= 
		\sigma x^\sigma \int_{(x,+\infty)} t^\sigma \int_t^\infty u^{-\sigma-1}\dd u\, \mu(\dd t) \\
		&= x^\sigma \bar{\mu}(x),
	\end{align*}
	which gives $(i)$.

	To show equivalence between $(ii)$ and $(iii)$ we use Theorem \ref{thm:tauber}. First take $k(x)= (p+2)x^{p+2}/(1+x)^{p+3}$. Then, it is clear that $\check{k}(z)$ is absolutely convergent for $|\Re z|<1$ and $\check{k}(z)\neq 0$ for $\Re z=0$. Moreover, $\check{k}(0)=1$. As before, define  $f(x)=\int_{[0,x]}t^{p+1}\mu(\dd t)$. Since $f$ is nondecreasing, by Theorem \ref{thm:tauber} $(ii)$, we see that
	\[
	f\in\Pi_L(p+1)\qquad \mbox{if and only if}\qquad  k\stackrel{M}{\ast}f\in \Pi_L(p+1).
	\]
	Finally,
	\[
	k\stackrel{M}{\ast}f(x) = \int_0^\infty k\left(\frac xt\right) f(t)\frac{\dd t}{t} = x^{p+2}\int_{[0,+\infty)} \frac{\dd f(t)}{(t+x)^{p+2}} = x^{p+2}\int_{[0,+\infty)} \frac{ t^{p+1}}{(t+x)^{p+2}}\mu(\dd t),
	\]
	where the second equality follows from integration by parts.
\end{proof}

Below we put into evidence that for $p=0$ in the theorem above we can replace $(iii)$ by another condition, which will turn out to be useful.
\begin{thm}\label{thm:taubA1}
	Let $L$ be a slowly varying function.
	The following assertions are equivalent:
	\begin{enumerate}[(i)]
		\item $\bar{\mu}(x)\sim x^{-1} L(x)$,
		\item 
		$x\mapsto \int_{[0,x]} t\,\mu(\dd t) \in \Pi_L(1)$,
		\item $x\mapsto x\int_{[0,+\infty)} \frac{ t}{t+x}\mu(\dd t) \in \Pi_L(1)$.
	\end{enumerate}
\end{thm}
\begin{proof}
	Equivalence between $(i)$ and $(ii)$ follows from Theorem \ref{thm:taubPi} with $p=0$.
	
	Define $f(x)=\int_{[0,x]}t\,\mu(\dd t)$ and $k(x)=x/(1+x)^2$. We have
	\[
	x\int_{[0,+\infty)} \frac{ t}{t+x}\mu(\dd t) =\int_{[0,+\infty)}\frac{x}{x+t}\dd f(t)=\int_0^\infty \frac{x}{(x+t)^2}f(t)\dd t = k\stackrel{M}{\ast}f(x).
	\]
	The Mellin transform $\check{k}$ of $k$ is absolutely convergent for $|\Re z|<1$ and $\check{k}(z)\neq 0$ for $\Re z=0$. Moreover, $\check{k}(0)=\int_0^\infty (1+t)^{-2}\dd t=1$.  Since $f$ is nondecreasing, by Theorem  \ref{thm:tauber} $(ii)$, we see that
	\[
	f\in\Pi_L(1)\qquad \mbox{if and only if}\qquad  k\stackrel{M}{\ast}f\in \Pi_L(1),
	\]
	which proves equivalence between $(ii)$ and $(iii)$.
\end{proof}

\begin{lemma}\label{lem:simpleT}\ 
	\begin{enumerate}[(i)]
		\item If  $f$ is positive and $f(x)\sim g(x)$ then, $k_p\stackrel{M}{\ast} f(x)\sim k_p\stackrel{M}{\ast} g(x)$.
		\item Let $L$ be a slowly varying function and assume that $f(x)=g(x)+c_0+o(x^{-\eps})$  for some $\eps>0$ and $c_0\in\R$. Then, $g\in \Pi_L(c)$ if and only if $f\in\Pi_L(c)$.
		\item Let $L$ be a slowly varying function and assume that $f(x)=g(x)(c_0+o(x^{-\eps}))$ for some $\eps>0$ and $c_0\in\R$. Then, $g\in \Pi_L(c)$ if and only if $f\in \Pi_L(c_0 c)$.
	\end{enumerate}
	\noindent	
\end{lemma}
\begin{proof}

		For $(i)$ fix $\eps>0$ and take $x_0>0$ such that $1-\eps\leq f(x)/g(x)\leq 1+\eps$ for $x>x_0$. Then, for $x>x_0$ we have
		\begin{multline*}
			(1-\eps)k_p \stackrel{M}{\ast} g(x) =x^{p+1}\int_x^\infty \frac{(1-\eps)g(t)}{t^{p+2}}\dd t \\
			\leq 
			x^{p+1}\int_x^\infty \frac{f(t)}{t^{p+2}}\dd t = 	k_p \stackrel{M}{\ast} f(x)
			\leq (1+\eps)	k_p \stackrel{M}{\ast} g(x).
		\end{multline*}
		
		For $(ii)$ and $(iii)$ we use the fact that for $\ell\in\mathcal{R}_0$ we have $\ell(x)=o(x^{\eps})$ for all $\eps>0$ and the result follows from point $(ii)$ of Theorem \ref{thm:PiL}.
\end{proof}

\section{S-transform}\label{sec:3}

In this section we study in detail analytic properties of Voiculescu's $S$-transform, seen as a function of a real argument. We put into evidence some facts about the expansion of the $S$-transform as the real argument $z\to 0^-$. While this section is mostly preliminary, it also contains a fair amount of new observations about the $S$-transform, which might be of independent interest.
Throughout this section we assume that $\mu\in\mathcal{M}_+$ and $\mu\neq\delta_0$. We shall also denote $\delta=\mu(\{0\})<1$.

\begin{defin}
	Suppose $\mu\in\mathcal{M}_+$ and assume
	that $\delta=\mu(\{0\})< 1$. The so-called moment transform of $\mu$ is defined as 
	\[
	\psi_\mu(z)= \int_{[0,+\infty)} \frac{z t}{1-z t}\mu(\dd t).
	\]
	The function $\psi_\mu\colon (-\infty,0)\to (\delta-1,0)$ is invertible (see \cite{BV93}) and we denote its inverse by $\chi_\mu$. 
	
	The $S$-transform of $\mu$ is defined by  
	\[
	S_\mu(z) = \frac{z+1}{z} \chi_\mu(z),\qquad z\in(\delta-1,0).
	\] 
\end{defin}
\begin{remark}
		Usually $S_\mu$ is defined for complex argument. Function $\psi_\mu$ is well defined for $z\in \mathbb{C}\setminus\R_+$ and it is univalent in the left-plane $i\mathbb{C}_+$. Then, $\chi_\mu\colon \psi_\mu(i\mathbb{C}_+)\to \mathbb{C}_+$ is the inverse of $\psi_\mu$ and $S_\mu(z) = (1+z^{-1}) \chi_\mu(z)$ for $z\in \psi_\mu(i\mathbb{C}_+)$. However, it is enough for us to work only with real functions.
		
		It is also worth noticing that the $S$-transform determines uniquely the probability measure, as it determines the moment transform $\psi$.
\end{remark}

In the proposition below we review some properties of the $S$-transform, which are relevant for us.
\begin{proposition}\label{thm:useful}\ 
	\begin{enumerate}[(i)]
		\item \cite[Proposition 6.1, 6.3]{BV93} $\psi_\mu$ and $\chi_\mu$ are are analytic in $(-\infty,0)$, respectively, $(\delta-1,0)$, hence $S_\mu$ is analytic in $(\delta -1,0)$.
		\item \cite[Proposition 6.8]{BV93} $S_\mu$ is decreasing on $(\delta-1,0)$ and positive. 
		\item \cite[Lemma 4]{HaM13} $S_\mu\big((\delta-1,0)\big) = \left(m_1^{-1}(\mu), m_{-1}(\mu) \right)$.
		\item \cite[Proposition 6.6]{BV93} Let $\mu,\nu\in\mathcal{M}_+$, none of them being $\delta_0$. Then we have $S_{\mu\boxtimes \nu}=S_\mu S_\nu$.
		\item \cite[Proposition 3.13]{HS07} If $\mu(\{0\})=0$, then $S_{\hat{\mu}}(z)=1/S_\mu(-1-z)$, $z\in(-1,0)$, where $\hat{\mu}$ is the pushforward measure of $\mu$ by the mapping $x\mapsto x^{-1}$.
	\end{enumerate}
\end{proposition}

We record for further reference a simple fact about the derivatives of the moment transform. Let as note that all derivatives of $\psi_\mu$ are monotonic.
\begin{lemma}\label{lem:psi1}
	Let $p\in\mathbb{N}\cup\{0\}$. For $z<0$ we have
	\[
	\psi_\mu^{(p+1)}(z)=(p+1)! \int_{[0,+\infty)} \frac{t^{p+1}}{(1-z t)^{p+2}}\mu(\dd t).
	\]
\end{lemma}

In the lemma below we find a useful formula for the $p$th derivative of the $S$-transform.
\begin{lemma}\label{lem:RepS}
	Let $p\in \mathbb{N}\cup\{0\}$. For $z\in(\delta-1,0)$ we have
	\begin{align*}
		S_\mu^{(p)}(z) & = \chi_\mu^{(p)}(z)+ \int_z^0 \frac{(-t)^{p}}{(-z)^{p+1}} \chi_\mu^{(p+1)}(t)\dd t.
	\end{align*}
\end{lemma}
\begin{proof}Recall that $S_\mu(z) = \chi_\mu(z)+ z^{-1}\chi_\mu(z)$, so it is enough to show that 
	\[
	\frac{\dd^p}{\dd z^p} \frac{\chi_\mu(z)}{z}= \int_z^0 \frac{(-t)^{p}}{(-z)^{p+1}} \chi_\mu^{(p+1)}(t)\dd t.
	\]
	We proceed by induction. For $p=0$, the right hand side above equals
	\[
	\frac{1}{-z}\int_z^0  \chi_\mu^\prime(t)\dd t = -\frac{1}{z} \left(\chi_\mu(0^-)-\chi_\mu(z)\right)= \frac{\chi_\mu(z)}{z},
	\]
	since $\chi_\mu(0^-)=0$. By the induction hypothesis we have
	\begin{align*}
		\frac{\dd^{p+1}}{\dd z^{p+1}} \frac{\chi_\mu(z)}{z}&=
		\frac{\dd}{\dd z} \frac{\dd^{p}}{\dd z^{p}} \frac{\chi_\mu(z)}{z} = \frac{\dd}{\dd z} 
		\frac{ \int_z^0 (-t)^{p} \chi_\mu^{(p+1)}(t)\dd t}{(-z)^{p+1}}\\
		& = \frac{ -  (-z)^{p} \chi_\mu^{(p+1)}(z)}{(-z)^{p+1}}+(p+1)\frac{ \int_z^0 (-t)^{p} \chi_\mu^{(p+1)}(t)\dd t}{(-z)^{p+2}}\\
		& = \int_z^0 \frac{(-t)^{p+1}}{(-z)^{p+2}} \chi_\mu^{(p+2)}(t)\dd t,
	\end{align*}
	where the last equality follows from integration by parts.
\end{proof}

\begin{remark}
	If $\delta=0$ and $m_{-1}(\mu)<+\infty$, then one can show that for $z\in(-1,0)$,
	\[
	\chi_\mu^{(p)}(z)=S_\mu^{(p)}(z)+ \frac{(-1)^{p+1}p!\, m_{-1}(\mu)}{(1+z)^{p+1}} -\int_{-1}^z \frac{(1+t)^p}{(1+z)^{p+1}} S_\mu^{(p+1)}(t)\dd t.
	\]
	Indeed, for arbitrary $\mu\in\mathcal{M}_+$ we  have
	\begin{align}
		\chi^{(p)}_\mu(z) & = S_\mu^{(p)}(z)-\frac{\dd^p}{\dd z^p} \left(S_\mu(z) \frac{1}{1+z}\right) =  S_\mu^{(p)}(z)-\sum_{n=0}^p \binom{p}{n} S_\mu^{(n)}(z) \frac{\dd^{p-n}}{\dd z^{p-n}}  \frac{1}{1+z} \nonumber\\
		& = S_\mu^{(p)}(z)-\sum_{n=0}^p \binom{p}{n} S_\mu^{(n)}(z)  (-1)^{p-n}\frac{(p-n)!}{(1+z)^{p-n+1}} \nonumber\\
		&=  S_\mu^{(p)}(z)+p! (-1-z)^{-p-1} \sum_{n=0}^p \frac{S_\mu^{(n)}(z)}{n!} (-1-z)^n. \label{eq:chip}
	\end{align}
	But if $\mu(\{0\})=0$ and  $S_\mu(-1^+)=m_{-1}(\mu)<+\infty$, then we have a one-sided Taylor expansion  with integral form of the remainder.
	\[
	S_\mu(-1^+)= \sum_{n=0}^p \frac{S_\mu^{(n)}(z)}{n!} (-1-z)^n + \frac{1}{p!}\int_z^{-1} (-1-t)^p S_\mu^{(p+1)}(t)\dd t.
	\]
\end{remark}

From the Lemma \ref{lem:RepS} we find another representation of the $p$th derivative of the $S$-transform. This representation allows us to apply standard Tauberian theorems. The kernel $k_p$ was defined in \eqref{def:kp}.
\begin{Corollary}\label{cor:RepS}
	For $x>(1-\delta)^{-1}$ we have
	\[
	S_\mu^{(p)}\left(-\frac1x\right)= \chi_\mu^{(p)}\left(-\frac1x\right)+k_p\stackrel{M}{\ast} \widetilde{\chi}_{p+1}(x),
	\]
	where $\widetilde{\chi}_{p+1}(x):=\chi_\mu^{(p+1)}(-1/x)$.
\end{Corollary}

In the next lemma we observe what an expansion of any of the three functions $(\psi_\mu,\chi_\mu,S_\mu)$ says about expansion of the other two functions.
\begin{lemma}\label{lem:ConvergenceS}
	
	Assume that $p\in\mathbb{N}$ and $\mu\in\mathcal{M}_+$.	Fix $\eps\in[0,1)$.
	The following three conditions are equivalent
	\begin{align}\label{eq:psiseries}
		\psi_\mu(z) = \sum_{n=1}^p m_n(\mu) z^n+o(z^{p+\eps}),\qquad \mbox{as }z\to0^-,
	\end{align}	
	\begin{align}\label{eq:chiseries}
		\chi_\mu(z) = \sum_{n=1}^p c_n z^n + o(z^{p+\eps}),\qquad \mbox{as }z\to 0^-
	\end{align}
	for some real coefficients $(c_n)_n$,
	\begin{align}\label{eq:Sseries}
		S_\mu(z) = \sum_{n=0}^{p-1} s_n z^n + o(z^{p-1+\eps}),\qquad \mbox{as }z\to 0^-
	\end{align}
	for some real coefficients $(s_n)_n$.

	Moreover if $m_p(\mu)<+\infty$ then all three equation \eqref{eq:psiseries}, \eqref{eq:chiseries} and \eqref{eq:Sseries} hold with $\eps=0$.
\end{lemma}
\begin{proof}
	Equivalence between \eqref{eq:chiseries} and \eqref{eq:Sseries} follows from the definition of $S_\mu$.	
	We shall show by induction that \eqref{eq:psiseries} implies \eqref{eq:chiseries}. The proof of the converse implication is similar. 
	
	First we note that if $m_1(\mu)<+\infty$, then  $\chi_\mu(z)\sim z/m_1(\mu)$ as $z\to0^-$, which follows immediately from
	\[
	m_1(\mu)=\lim_{z\to 0^-}\frac{\psi_\mu(z)}{z} = \lim_{t\to 0^-}\frac{t}{\chi_\mu(t)},
	\]
	where we have made the substitution $z=\chi_\mu(t)$. This in particular implies that if $m_1(\mu)<+\infty$, then  $o(\chi_\mu(z)^{\alpha})$ coincides with $o(z^{\alpha})$ for all $\alpha\in\R$.
	
	Substituting $z\mapsto \chi_\mu(z)$ in \eqref{eq:psiseries} with $p=1$ we get
	\[
	z = m_1(\mu) \chi_\mu(z)+o(\chi_\mu(z)^{1+\eps})=m_1(\mu) \chi_\mu(z)+o(z^{1+\eps}),
	\]
	which is \eqref{eq:chiseries}. Suppose that \eqref{eq:psiseries} implies \eqref{eq:chiseries} for expansions of order $p$. Take the expansion \eqref{eq:psiseries} of order $p+1$. Plugging $z\mapsto \chi_\mu(z)$ into this expansion gives us
	\begin{align}\label{eq:chiseries2}
		z = m_1(\mu) \chi_\mu(z) + \sum_{n=2}^{p+1} m_n(\mu) \chi_\mu(z)^n+o(z^{p+1+\eps}).
	\end{align}
	By the induction hypothesis, we have that expansion \eqref{eq:chiseries} of order $p$ holds. This implies that for $n\geq 2$ we have 
	\[
	\chi_\mu(z)^n = \left(\sum_{k=1}^p c_k z^k + o(z^{p+\eps})\right)^n = \left(\sum_{k=1}^p c_k z^k \right)^n + o(z^{p+1+\eps}).
	\]
	With such substitution, \eqref{eq:chiseries2} gives \eqref{eq:chiseries}.	
\end{proof}

We conclude this section with a lemma which gives a direct connection between derivatives of $\psi_\mu$ and $\chi_\mu$.
\begin{lemma}\label{lem:Bruno}\ 
	Let $\mu\in\mathcal{M}_+$ and $p\in\mathbb{N}$. 
	\begin{enumerate}[(i)]
		\item For $z\in(\delta-1,0)$ we have
		\begin{align}\label{eq:Bruno}
			\psi_\mu^{(p+1)}(\chi_\mu(z)) \chi_\mu^\prime(z)^{p+1} + \psi_\mu^\prime(\chi_\mu(z)) \chi_\mu^{(p+1)}(z) = Q(z)
		\end{align}
		where $Q$ on the right hand side is a polynomial in $\psi_\mu^{(l)}(\chi_\mu(z))$ and $\chi_\mu^{(l)}(z)$ for $l=1,\ldots,p$.
		\item If $m_p(\mu)<+\infty$, then the limit
		\begin{align}\label{eq:BrunoLimit}
			\lim_{z\to0^-}\left(\psi_\mu^{(p+1)}(\chi_\mu(z)) \chi_\mu^\prime(z)^{p+2} +  \chi_\mu^{(p+1)}(z)\right)
		\end{align} 
		exists and is finite.
	\end{enumerate}
\end{lemma}
\begin{proof}
	\begin{enumerate}[$(i)$]
		\item The first equality follows from the Fa\'a di Bruno formula, which gives for $p\in\mathbb{N}$,
		\[
		0 = \frac{\dd^{p+1}}{\dd z^{p+1}} \psi_\mu\left(\chi_\mu(z)\right) = \sum_{\pi\in\Pi} \psi_\mu^{(|\pi|)}(\chi_\mu(z))\prod_{B\in\pi}\chi_\mu^{(|B|)}(z),
		\]
		where $\pi$	runs through the set $\Pi$ of all partitions of the set $\{1,\ldots, p+1 \}$ and $B\in\pi$ runs through the list of all of the blocks of the partition $\pi$. The two terms on the left hand side of \eqref{eq:Bruno} correspond to partitions $\{\{1\},\ldots,\{p+1\}\}$ and  $\{\{1,\ldots,p+1\}\}$. The right hand side of \eqref{eq:Bruno} contains all remaining terms.
		\item  Since $\psi_\mu^\prime(\chi_\mu(z))=1/\chi_\mu^\prime(z)$, \eqref{eq:BrunoLimit} follows from $(i)$ and
		the fact that all $\psi_\mu^{(l)}(\chi_\mu(z))$ and $\chi_\mu^{(l)}(z)$, $l=1,\ldots,p$, have finite limits as $z\to0^-$.
	\end{enumerate}
\end{proof}

\section{$S$-transform and the tail at $+\infty$}\label{sec:4}

Our main tool concerns the relation the right tail of a measure and its $S$-transform. The results of this section are the workhorse of this paper and all results that we prove in the subsequent sections rely on them. Our work allows to answer the open question posed in \cite[\S2.4 (1) and (2)]{CH18}.

We have that (see Lemma \ref{lem:ConvergenceS}), if $\mu\in\mathcal{M}_p$, $p\in\mathbb{N}\cup\{0\}$,
then
\[
S_\mu(z) = \sum_{n=0}^{p-1} s_n z^n + r(z),\qquad z\in(\mu\{0\}-1,0),
\]
where $(s_n)_n$ are real numbers and each $s_k$ is a rational function of moments of $\mu$ up to $k+1$. Moreover, $r(z)=o(z^{p-1})$ as $z\to0^-$. It turns out that the regular variation of $\bar{\mu}$ is equivalent to regular variation (or its relatives) of function $r$. However, instead of working with $r$, we decided to describe this correspondence in terms of its $p$th derivative, $S_\mu^{(p)}=r^{(p)}$. As all involved functions are analytic, there are no regularity issues. Such approach allows to eliminate the ``Taylor polynomial'' $\sum_{n=0}^{p-1} s_n z^n$ and  avoid dealing with complicated combinatorics of coefficients in the series expansion of the $S$-transform.

Recall that we say that a measure $\mu$ has regularly varying tail with index $\alpha$ when
\begin{align}\label{eq:L}
	\bar{\mu}(x)=\mu\big((x,+\infty)\big)\sim \frac{L(x)}{x^\alpha}.
\end{align}

The first result concerns probability measures in $\mathcal{M}_0$, that is, measures whose first moment is infinite.

\begin{thm}\label{thm:a0}
	Let $\alpha=0$ and assume that $L\in \mathcal{R}_0$. If \eqref{eq:L} holds, then 
	\begin{align}\label{eq:SinKR}
	x\mapsto \frac{1}{S_\mu(-1/x)} \in K\mathcal{R}_\infty.
	\end{align}
	Conversely, if \eqref{eq:SinKR} is satisfied, then \eqref{eq:L} holds with
	\[
	L(x)=1/f_\mu^{\langle-1\rangle}(x),
	\]
	where 
	\[
	f_\mu(x):=-\frac{1}{\chi_\mu(-1/x)}=\frac{x-1}{S_\mu(-1/x)}.
	\]
\end{thm}

\begin{remark}
	Suppose that $L_1, L_2\in \mathcal{R}_0$ are decreasing to $0$. Then, condition $L_1(x)\sim L_2(x)$ does not imply in general that $L_1^{\langle-1\rangle}(t)\sim L_2^{\langle-1\rangle}(t)$ as $t\to 0^+$. E.g. consider $L_1(x)=1/\log(x)$ and $L_2(x)=1/\log(1+x)$. Thus, in Theorem \ref{thm:a0} we have $L(x)\sim 1/f_\mu^{\langle-1\rangle}(x)$ while in general it is not true that $L^{\langle-1\rangle}(t)\sim f_\mu(1/t)$ as $t\to 0^+$.
\end{remark}

\begin{remark}
	If in Theorem \ref{thm:a0} we consider $\bar{\mu}\in \Pi_L(1)$ instead of $\bar{\mu}\in \mathcal{R}_0$, then one can show the function $f_\mu$ belongs to a subclass $\Gamma$ of $K\mathcal{R}_\infty$, \cite[Section 3.10]{BGT89}.
\end{remark}

\begin{proof}[Proof of Theorem~\ref{thm:a0}]
	For any $x>0$ we have 
	\begin{align*}
		-\psi_\mu\left(-1/x\right)&=\int_{[0,+\infty)} \frac{t}{t+x}\mu(\dd t)= 1-\int_{[0,+\infty)} \frac{x}{t+x}\mu(\dd t) \\
		&= \int_0^\infty \frac{x}{(t+x)^2}\bar{\mu}(t)\dd t=k\stackrel{M}{\ast}\bar{\mu}(x),
	\end{align*}	
	where $k(x)=x/(1+x)^2$. The Mellin transform of $k$ is absolutely convergent for $|\mathfrak{R}z|<1$ and $\check{k}(0)=1$. Since $\bar{\mu}$ is monotonic and $x \bar{\mu}(x)$ is bounded on intervals $(0,a]$, by Theorem \ref{thm:tauber} $(i)$, we see that 
	\eqref{eq:L} is equivalent to
	\[
	-\psi_\mu\left(-1/x\right)=\int_{[0,+\infty)} \frac{t}{t+x}\mu(\dd t) \sim L(x).
	\]
	Thus, if we assume \eqref{eq:L}, we have
	\[
	-\frac{1}{\psi_\mu\left(-1/x\right)}\sim\frac{1}{L(x)}\in \mathcal{R}_0.
	\]
	Further, by Theorem \ref{thm:KR}, the inverse of $x\mapsto-1/\psi_\mu(-1/x)$, which is $x\mapsto-1/\chi_\mu(-1/x)$ belongs to $K\mathcal{R}_{\infty}$. 
	We have
	\[
	-\frac{1}{\chi_\mu(-1/x)} = \frac{x-1}{S_\mu\left(-1/x\right)}
	\]
	and so, by Remark \ref{rem:KR}, $x\mapsto1/ S_\mu\left(-1/x\right)$ belongs to $K\mathcal{R}_\infty$ as well.
	
	Since all steps above can be reversed, the prove is complete.
\end{proof}

\begin{Corollary}\label{cor:expv}
	Assume that $\mu\in\mathcal{M}_+$ is such that $S_\mu(-1/x)=\exp(-s(x))$, where $s\in\mathcal{R}_\rho$ with $\rho>0$. Then,
	\[
	\bar{\mu}(x)\sim 1/s^{\langle-1\rangle}(\log(x)).
	\] 
	In particular, if $L\in\mathcal{R}_0$ and $s(x)=x^\rho L^\rho(x)$, then
	\[
	\bar{\mu}(x) \sim 1/\left(\log(x)^{1/\rho} L^\#(\log(x)^{1/\rho})\right).
	\]
\end{Corollary}
\begin{proof}
	We will start by showing that our assumptions fall in the framework of Theorem \ref{thm:a0}, i.e. \eqref{eq:SinKR} holds. Clearly, we have $x\mapsto1/S_\mu(-1/x)=\exp(s(x))$ belongs to $\mathcal{R}_\infty$. Moreover, it is known that $f_\mu(x)=1/\chi_\mu(-1/x)=(x-1)/S_\mu(-1/x)$ is nondecreasing and clearly also belongs to $\mathcal{R}_\infty$. Thus, by Theorem \ref{thm:KR} $(ii)$ we see that $f_\mu$ belongs to $K\mathcal{R}_\infty$.  Remark \ref{rem:KR} implies \eqref{eq:SinKR}. 
	
	We have
	\[
	f_\mu(x) = e^{s(x)(1+o(1))},
	\]
	which implies that
	\[
	f_\mu^{\langle-1\rangle}(t)\sim s^{\langle-1\rangle}(\log(t)). 
	\]
	Thus, by Theorem \ref{thm:a0}, we obtain the first part of the assertion. The second part follows from Theorem \ref{thm:conj}.
\end{proof}

\begin{thm}\label{thm:01}
	Let $\alpha\in(0,1)$ and $L\in \mathcal{R}_0$. Then, \eqref{eq:L} is equivalent to
	\begin{align}
	\label{eq:S01}
	S_{\mu}\left(-\frac1x\right)&\sim \left( \frac{\sin(\pi \alpha)}{\pi \alpha} \right)^{1/\alpha}\frac{x^{1-1/\alpha}}{ M^\#(x^{1/\alpha})},
	\end{align}
	where $M^\#$ is a de Bruijn conjugate of a slowly varying function $M:=L^{-1/\alpha}$.
\end{thm}

\begin{remark}
If \eqref{eq:L} holds with $L(x) = c \prod_{k=1}^n (\log^{(\circ k)}(x))^{\alpha_k}$ (recall Theorem \ref{thm:conj}), then \eqref{eq:S01} may be replaced by
\begin{align*}
	S_{\mu}\left(-\frac1x\right)\sim \left( \frac{\sin(\pi \alpha)}{\pi \alpha} \right)^{1/\alpha} \frac{x^{1-1/\alpha}} {L^{1/\alpha}(x^{1/\alpha})}.
\end{align*}
In particular, if $L(x)\sim c>0$, then \eqref{eq:S01} is equivalent to
	\begin{align}\label{eqn:4.4}
	S_{\mu}\left(-\frac1x\right)\sim c^{-1/\alpha}\left( \frac{\sin(\pi \alpha)}{\pi \alpha} \right)^{1/\alpha}x^{1-1/\alpha}.
	\end{align}
\end{remark}

\begin{proof}[Proof of Theorem~\ref{thm:01}]
	With $U(x):=\int_{[0,x]} t\,\mu(\dd t)$, $x\geq0$, by Corollary \ref{cor:taub0}, we have that \eqref{eq:L} is equivalent to 
	\[
	U(x)\sim  \frac{\alpha}{1-\alpha} x^{1-\alpha}L(x).
	\]
	By Theorem \ref{thm:TaubSt} with $\sigma=1-\alpha, \rho=1$ and $c=\alpha/(1-\alpha)$, the above asymptotic behavior is equivalent to
	\[
	\int_{[0,+\infty)} \frac{\dd U(t)}{t+x}\sim \alpha \Gamma(\alpha)\Gamma(1-\alpha)x^{-\alpha}L(x).
	\]
	Note that we used there the basic identity $\Gamma(2-\alpha)=(1-\alpha)\Gamma(1-\alpha)$.
	By Euler's reflection formula we have
	\[
	\alpha \Gamma(\alpha)\Gamma(1-\alpha) = \frac{\pi\alpha}{\sin(\pi\alpha)}
	\]
	and thus
	\[
	-\psi_\mu\left(-\frac1x\right) = \int_{[0,+\infty)} \frac{t}{t+x}\mu(\dd t) = \int_{[0,+\infty)} \frac{\dd U(t)}{t+x} \sim \frac{\pi\alpha}{\sin(\pi\alpha)} x^{-\alpha}L(x).
	\]
	Then, with $d:=\sin(\pi\alpha)/(\pi \alpha)$ and $M:=L^{-1/\alpha}\in\mathcal{R}_0$, the right hand side above equals $1/(d\,x^{\alpha}  M^\alpha(x))$. Observe that $-1/\psi_\mu(-1/x)$ asymptotically behaves like $x\mapsto d\,x^{\alpha}  M^\alpha(x)$. Theorem \ref{thm:conj} allows us to determine the asymptotic inverse and we get that the 
	\[
	x\mapsto d^{-1/\alpha} x^{1/\alpha}M^\#(x^{1/\alpha}).
	\]
	Taking into account that the compositional inverse of $-1/\psi_\mu(-1/x)$ is $-1/\chi_\mu(-1/x)$, we get
	\[
	\chi_\mu\left(-\frac1x\right)\sim -d^{1/\alpha}\frac{x^{-1/\alpha}}{ M^\#(x^{1/\alpha})}.
	\]
	Since $-\chi_\mu(-1/x) = S_\mu(-1/x)/(x-1)$, we have equivalently
	\[
	S_\mu\left(-\frac1x\right)\sim d^{1/\alpha}\frac{x^{1-1/\alpha}}{ M^\#(x^{1/\alpha})}.
	\]
\end{proof}

Cases $\alpha<1$ and $\alpha>1$ (see Theorem \ref{thm:p} below) are very different. This is due to the fact that under \eqref{eq:L} for $\alpha<1$ we have $m_1(\mu)=+\infty$ and for $\alpha>1$ the first moment is finite. This two regimes correspond to different characters of corresponding $S$-transforms. A bit surprising, the case $\alpha=1$, despite the fact whether $m_1(\mu)$ is finite or infinite, is treated jointly in a theorem below. However, in the case $m_1(\mu)<+\infty$, its formulation can be considerably simplified, see Remark \ref{rem:8}.
\begin{thm}\label{thm:alpha1}
	Let $\alpha=1$ and $L\in \mathcal{R}_0$. 
	Then \eqref{eq:L} i.e. $\bar{\mu}(x)\sim x^{-1}L(x)$,
	 is equivalent to
	\begin{align}\label{eq:a1}
	x\mapsto\frac{1}{S_\mu\left(-1/x\right)}\in \Pi_M(1),
	\end{align}
	where $M(x):=L\big((x-1)/S_\mu(-1/x)\big)\in \mathcal{R}_0$.
\end{thm}
\begin{remark}\label{rem:8}
Consider the case $m_1(\mu)<+\infty$. Then $1/S_\mu(-1/x)\to m_1(\mu)$ as $x\to+\infty$ and by the slow variation of $L$ we have $M(x)\sim L(x)$. Whence $\Pi_M(1)=\Pi_L(1)$. Further, in view of Remark \ref{rem:PiInv},  $x\mapsto1/S_\mu(-1/x)\in\Pi_L(1)$ is equivalent to 
\begin{align}\label{eq:todiff}
x\mapsto S_\mu\left(-\frac1x\right)\in\Pi_{L(\cdot)S_\mu\left(-1/\cdot\right)^2}(-1) = \Pi_L\left(-\frac{1}{m_1(\mu)^2}\right).
\end{align}
If $S_\mu^\prime$ is monotonic, we can use the Monotone Density Theorem (see \cite[Theorem 3.6.8]{BGT89}) and ``differentiate'' \eqref{eq:todiff}. In this way we prove the following reformulation of the condition in Theorem \ref{thm:alpha1}: \eqref{eq:a1} is equivalent to 
\[
S_\mu^{\prime}\left(-\frac1x\right) \frac{1}{x^2}\sim -\frac{1}{m_1(\mu)^2} \frac{L(x)}{x}.
\]
This equivalence will be proved as a special case of Theorem \ref{thm:p}.
\end{remark}
\begin{proof}[Proof of Theorem \ref{thm:alpha1}]
	By Theorem \ref{thm:taubA1}, the condition \eqref{eq:L} is equivalent to
	\begin{align}\label{eq:new1}
		-x \,\psi_\mu\left(-\frac 1x\right) = x\int_{[0,+\infty)} \frac{t}{t+x}\mu(\dd t)\in \Pi_L(1).
	\end{align}
	We shall show only that \eqref{eq:new1} implies \eqref{eq:a1}. The proof of the converse implication is similar.

	Consider the function $f(x)=-1/\left(x\,\psi_\mu(-1/x)\right)$. Since $1/f\in\Pi_L(1)$,  Remark \ref{rem:PiInv} implies that $f\in \Pi_{M_1}(-1)$ with
	\[
	M_1(x) = \frac{L(x)}{x^2 \psi_\mu(-1/x)^2}.
	\] 
	By definition of $f$ we have
	\begin{align*}
		-\frac{1}{\psi_\mu(-1/x)} =x f(x).
	\end{align*}
	Taking inverses of both sides of the above equation, gives
	\[
	-\frac{1}{\chi_\mu(-1/x)}=x f^\ast(x).
	\]
	where $f^\ast$ is $\Pi$-conjugate function for $f$ in the sense of Definition \ref{def:conj}. By Theorem \ref{thm:PiConj} $(i)$, we have $f^\ast\in \Pi_{M_2}(1)$, where for $M_2$ after some transformations we obtain 
	\begin{align*}
		M_2(x)=  M_1(x f^\ast(x)) f^\ast(x)^2=  L\left(-\frac{1}{\chi_\mu(-1/x)}\right).
	\end{align*}
	By the definition of the $S$-transform we have
	\[
	\frac{1}{S_\mu\left(-1/x\right)}=(x-1)^{-1}\frac{-1}{\chi_\mu\left(-1/x\right)}=\frac{x f^\ast(x)}{x-1} \sim_{\Pi_{M_2}} f^\ast(x)\in \Pi_{M_2}(1),
	\]
	since $x/(x-1)=1+o(x^{-\eps})$, $\varepsilon\in(0,1)$. 
\end{proof}

If the measure $\mu$ has the first moment (i.e. $m_1(\mu)<+\infty$) and $\bar{\mu}$ is regularly varying, then the behavior of the $S$-transform of $\mu$ is more delicate and different from the case $m_1(\mu)=+\infty$. Remarkably, independently of the value of $\alpha$ in $\eqref{eq:L}$, the asymptotic behavior of the $p$th derivative of the $S$-transform depends on the first moment $m_1(\mu)$ and and no other moments are involved. This observation will have some deeper consequences in Section \ref{sec:5}, Theorem \ref{thm:pat0}, where we investigate the behavior of the tail of the free multiplicative convolution.

\begin{thm}\label{thm:p}
	Let $p\in\mathbb{N}$, $\alpha\in[p,p+1]$ and $L\in \mathcal{R}_0$. 
	
(i) If one of the following conditions holds
\begin{enumerate}
	\item[(a)] $\alpha\in(p,p+1)$, or
	\item[(b)] $\alpha=p$ and $\int_1^{+\infty} L(t)/t\,\dd t<+\infty$,
\end{enumerate} 
then \eqref{eq:L}, i.e. $\bar{\mu}(x)\sim x^{-\alpha}L(x)$,
 is equivalent to 
\begin{align}\label{eq:Sp}
S_{\mu}^{(p)}\left(-\frac1x\right)\sim -\frac{\Gamma(\alpha+1)\Gamma(p+1-\alpha)}{m_1(\mu)^{\alpha+1}} x^{p+1-\alpha}L(x).
\end{align}

(ii) If $\alpha=p+1$ and $\int_1^{+\infty} L(t)/t\,\dd t=+\infty$, then \eqref{eq:L} is equivalent to 
\begin{align}\label{eq:Sp2}
x\mapsto S_{\mu}^{(p)}\left(-\frac1x\right) \in \Pi_L\left(-\frac{(p+1)!}{m_1(\mu)^{p+2}}\right).
\end{align}
\end{thm}

\begin{remark}
	Recall the notion of one-sided Taylor expansion from Lemma \ref{lem:expansion}.
	Assume that \eqref{eq:Sp} holds with $\alpha\in(p,p+1)$. Then, the remainder term $r(z)$ in series expansion of $S_\mu(z)=\sum_{n=0}^{p-1}s_n z^n + r(z)$ is asymptotically equivalent to (use e.g. \cite[Proposition 1.5.10]{BGT89})
	\begin{align*}
	(-1)^{p+1}\frac{\Gamma(1+\alpha)\Gamma(p+1-\alpha)}{m_1(\mu)^{\alpha+1}\prod_{k=1}^p(\alpha-k)}(-z)^{\alpha-1}L\left(-\frac1z\right)	
	=\frac{\pi\alpha}{m_1(\mu)^{\alpha+1}\sin\left(\pi\alpha\right)}(-z)^{\alpha-1}L\left(-\frac1z\right),
	\end{align*}
	where we have used Euler's reflection formula
	\[
	\Gamma(p+1-\alpha) = \frac{\pi}{\sin\left(\pi(p+1-\alpha)\right)}\frac{1}{\Gamma(\alpha-p)}=(-1)^{p+1}\frac{\pi}{\sin\left(\pi\alpha\right)}\frac{1}{\Gamma(\alpha-p)}
	\]
	and standard properties of the gamma function.
\end{remark}
\begin{proof}[Proof of Theorem~\ref{thm:p}]
	We start with $(i)$ and present only the proof from \eqref{eq:Sp} to \eqref{eq:L}. The proof of the converse implication goes along the same steps.
	
	First, let us note that thanks to Remark \ref{rem:ap} condition \eqref{eq:L} under our assumptions implies that $m_p(\mu)<+\infty$.  Moreover, under considered assumptions, condition \eqref{eq:Sp} alone also implies $m_p(\mu)<+\infty$.  Indeed, if $\alpha\in(p,p+1)$, then under \eqref{eq:Sp} we have $S_\mu^{(p)}(z)=o(z^{-\varepsilon})$ for $\varepsilon\in(1+p-\alpha,1)$. Then, 
Lemma \ref{lem:expansion} with $f=S_\mu$ and $n=0$ implies \eqref{eq:Sseries}, that is,
\[
S_\mu(z) = \sum_{n=0}^{p-1} s_n z^n + o(z^{p-1+\eps}),\qquad \mbox{as }z\to 0^-
\]
for some real coefficients $(s_n)_n$.
		If $\alpha=p$, thanks to the assumption $\int_1^{+\infty} L(t)/t\,\dd t<+\infty$, 
		we can still integrate \eqref{eq:Sp} to obtain \eqref{eq:Sseries}.
		Thus, by Lemma \ref{lem:ConvergenceS} we obtain $m_p(\mu)<+\infty$ for $\alpha\in[p,p+1)$.

	Recall that kernel $k_p$ is defined in \eqref{def:kp} and we denote $\widetilde{\chi}_{p+1}(x):=\chi_\mu^{(p+1)}(-1/x)$.
	
	\textsc{Claim 1.} $k_p\stackrel{M}{\ast} \widetilde{\chi}_{p+1}(x) \sim -{\Gamma(\alpha+1)\Gamma(p+1-\alpha)}{m_1(\mu)^{-1-\alpha}} x^{p+1-\alpha} L(x).$\\
	The r.h.s. of \eqref{eq:Sp} diverges to infinity. Moreover, $m_p(\mu)<+\infty$ implies that $\chi_\mu^{(p)}(-1/x)$ has a finite limit as $x\to+\infty$. Thus, the claim follows from Corollary \ref{cor:RepS}.
	
	\textsc{Claim 2.} $k_p\stackrel{M}{\ast} \widetilde{\chi}_{p+1}(x)\sim -k_p\stackrel{M}{\ast}\widetilde{\psi}_{p+1}(x)$ for $\widetilde{\psi}_{p+1}(x) := \psi_\mu^{(p+1)}(\chi_\mu(-1/x)) \chi_\mu^\prime(-1/x)^{p+2}$.\\
	For any $c\in\R$ we have $k_p\stackrel{M}{\ast}(c+o(1))=c+o(1)$. Thus, Lemma \ref{lem:Bruno} $(ii)$ implies that
	\[
	k_p\stackrel{M}{\ast}\widetilde{\psi}_{p+1}(x) = -k_p\stackrel{M}{\ast} \widetilde{\chi}_{p+1}(x) + c+o(1)
	\]
	for some $c\in\R$. 
	
	\textsc{Claim 3.} $k_p\stackrel{M}{\ast}\widetilde{\psi}_{p+1}(x) \sim  m_1(\mu)^{-p-2} k_p\stackrel{M}{\ast}\psi_\mu^{(p+1)}\left(\chi_\mu\left(-1/\cdot\right)\right)(x).$\\
	The claim follows from Lemma \ref{lem:simpleT} $(i)$ and the fact that 
	\[
	\widetilde{\psi}_{p+1}(x)\sim \frac{1}{m_1(\mu)^{p+2}} \psi_\mu^{(p+1)}\left(\chi_\mu\left(-\frac1x\right)\right).
	\]
	
	\textsc{Claim 4.} $\psi_\mu^{(p+1)}\left(\chi_\mu\left(-1/x\right)\right) 
	\sim  \alpha \Gamma(\alpha+1)\Gamma(p+1-\alpha) \left(m_1(\mu) x\right)^{p+1-\alpha} L(x).$\\
	Denote $f(x) =\psi_\mu^{(p+1)}\left(\chi_\mu\left(-1/x\right)\right)$. The three claims above give that
	\begin{align*}
		k_p\stackrel{M}{\ast}f(x) &\sim -  m_1(\mu)^{p+2} k_p\stackrel{M}{\ast}\widetilde{\chi}_{p+1}(x) 
	\end{align*}
	and that the r.h.s above is regularly varying with index $p+1-\alpha$.
	Since $\psi_\mu^{(p+1)}$ and $\chi_\mu$ are both monotonic, function  $f$ is also monotonic (by Lemma \ref{lem:psi1}). Thus, by Theorem \ref{thm:taub0} with $\gamma = p+1-\alpha$, we have 
	\begin{align*}
		\psi_\mu^{(p+1)}\left(\chi_\mu\left(-\frac1x\right)\right) =f(x)
		\sim  - m_1(\mu)^{p+2} \frac{k_p\stackrel{M}{\ast}\widetilde{\chi}_{p+1}(x)}{\check{k}_p(p+1-\alpha)}.
	\end{align*}
	We obtain the assertion after plugging \eqref{eq:kpcheck} and using Claim 1.
	
	\textsc{Claim 5.} $\psi_\mu^{(p+1)}\left(-1/y\right)\sim \alpha \Gamma(\alpha+1)\Gamma(p+1-\alpha) y^{p+1-\alpha}L(y).$\\
	Let us substitute $-1/x = \psi_\mu(-1/y)$ into Claim 4. Observe that we have $x\to+\infty$ if and only if $y\to+\infty$. Moreover since $\psi_\mu(-1/y)\sim -m_1(\mu)/y$ we get that
	\[
	m_1(\mu)x = - \frac{m_1(\mu)}{\psi_\mu(-1/y)} \sim y
	\]
	and by the Uniform Convergence Theorem \cite[Theorem 1.2.1]{BGT89} we obtain
	\[
	L(x) = L\left( -\frac{1}{\psi_\mu(-1/y)}\right)\sim L(y).
	\]

	\textsc{Claim 6.} Condition \eqref{eq:L} holds.\\
	By Lemma \ref{lem:psi1}, we have for $y>0$,
	\begin{align}\label{eq:Psipp1}
		\psi_\mu^{(p+1)}\left(-\frac1y\right)= (p+1)!\, y^{p+2}\,\int_{[0,+\infty)}  \frac{ t^{p+1}}{(t+y)^{p+2}}\mu(\dd t).
	\end{align}
	Thus using Claim 5 and Corollary \ref{cor:taub0} we finally get \eqref{eq:L}.

	The proof of $(ii$) works similarly, but the technical details are a bit more complicated.
	First observe that similarly as in $(i)$, both \eqref{eq:L} and \eqref{eq:Sp2} under $(ii)$ imply $m_p(\mu)<+\infty$. The implication from \eqref{eq:Sp2} follows from the fact that for $c>0$ we have $\Pi_L(c)\subset \mathcal{R}_0$ and thus $S_\mu^{(p)}(-1/x)\sim {-} \ell\left(x\right)=o(x^{\varepsilon})$ for $\ell\in \mathcal{R}_0$ and all $\varepsilon>0$. Hence, the argument used in $(i)$ works here as well.

	Again, we present only the harder part of the proof, which goes from \eqref{eq:Sp2} to \eqref{eq:L}.
	
	\textsc{Claim 1.} $k_p\stackrel{M}{\ast}\widetilde{\chi}_{p+1}\in \Pi_L\left(-(p+1)!m_1(\mu)^{-p-2}\right)$.\\

From Corollary \ref{cor:RepS} we have 
\[
S_\mu^{(p)}\left(-\frac1x\right)= \chi_\mu^{(p)}\left(-\frac1x\right)+k_p\stackrel{M}{\ast} \widetilde{\chi}_{p+1}(x).
\]
In view of Lemma \ref{lem:simpleT} $(ii)$ it is enough to show that \eqref{eq:Sp2} implies
\begin{align}\label{eq:chipnew}
	\chi_\mu^{(p)}(-1/x)=\chi_\mu^{(p)}(0^-)+o(x^{-\eps})
\end{align}
for some $\eps>0$. 
Condition \eqref{eq:Sp2} implies that for $\ell(x):=-S_\mu^{(p)}(-1/x)$ we have $\ell\in\mathcal{R}_0$. By Lemma \ref{lem:expansion} we have for $n\in\{0,\ldots,p-1\}$, $S_\mu^{(n)}(z)=S_\mu^{(n)}(0^-)+o(z^\eps)$.
After rearranging \eqref{eq:chip} we get 
\begin{align*}
	\chi^{(p)}_\mu(z) 
	=\frac{z}{1+z}S_\mu^{(p)}(z)+p! \sum_{n=0}^{p-1} \frac{S_\mu^{(n)}(z)}{n!} (-1-z)^{n-p-1}
\end{align*}
Clearly, all terms above are of the form $c+o(z^\varepsilon)$ for some $c\in\R$ and any $\varepsilon\in(0,1)$. Hence, we have \eqref{eq:chipnew} and Claim 1 follows.

	\textsc{Claim 2.} $k_p\stackrel{M}{\ast}\widetilde{\psi}_{p+1}\in \Pi_L\left((p+1)!m_1(\mu)^{-p-2}\right),$ \\where again  $\widetilde{\psi}_{p+1}(x) := \psi_\mu^{(p+1)}(\chi_\mu(-1/x)) \chi_\mu^\prime(-1/x)^{p+2}$.\\

	By Lemma \ref{lem:Bruno} $(i)$ we have 
		\[
	\widetilde{\chi}_{p+1}(x)+\widetilde{\psi}_{p+1}(x)=Q(-1/x),  
	\]
	where $Q$ is a polynomial in $\psi_\mu^{(l)}(\chi_\mu(-1/x))$ and $\chi_\mu^{(l)}(-1/x)$ for $l=1,\ldots,p$. Arguing as in Claim~1, we infer that $\psi_\mu^{(l)}(z)$ and $\chi_\mu^{(l)}(z)$ are all of the form $c+o(z^\varepsilon)$ for $\varepsilon\in(0,1)$. Under these circumstances, \eqref{eq:BrunoLimit} can be strengthened to
	\[
	\widetilde{\chi}_{p+1}(x)+\widetilde{\psi}_{p+1}(x)=c+o(x^{-\eps}),
	\]
	where $c\in\R$.  Claim~1 and Lemma \ref{lem:simpleT} $(ii)$  conclude the proof of this claim.

	\textsc{Claim 3.}
	$k_p\stackrel{M}{\ast}\psi_\mu^{(p+1)}\left(\chi_\mu\left(-1/\cdot\right)\right)\in  \Pi_L\left((p+1)!\right)$.\\
	Eq. \eqref{eq:chipnew} which we already established, upon repeated integration, implies that 
	\[
	\chi_\mu^\prime\left(-\frac1x\right)= \chi_\mu^{\prime}(0^-)+o(x^{-\eps}) =  \frac{1}{m_1(\mu)}+o(x^{-\eps}).
	\]
	Thus, here we get
	\[
	\widetilde{\psi}_{p+1}(x)=\psi_\mu^{(p+1)}\left(\chi_\mu\left(-\frac1x\right)\right)\left(\frac{1}{m_1(\mu)^{p+2}}+o(x^{-\eps}) \right).
	\]
	Now we use Lemma \ref{lem:simpleT} $(iii)$.
	
	\textsc{Claim 4.} $\psi_\mu^{(p+1)}\left(\chi_\mu\left(-1/\cdot\right)\right) \in \Pi_L\left((p+1)! (p+1)\right)$.\\
	Since $\psi_\mu^{(p+1)}(\chi_\mu(-1/\cdot))$ is monotonic, by Theorem \ref{thm:tauber} $(ii)$ applied to Claim 3 we infer that 
	\[
	\psi_\mu^{(p+1)}\left(\chi_\mu\left(-\frac1\cdot\right)\right) \in \Pi_L\left(\frac{(p+1)!}{\check{k}(0)}\right).
	\]
	By \eqref{eq:kpcheck} we have $\check{k}(0)=1/(1+p)$.
	
	\textsc{Claim 5.} $\psi_\mu^{(p+1)}\left(-1/\cdot\right) \in \Pi_L\left((p+1)! (p+1)\right)$.\\
	By Claim 4 and Theorem \ref{thm:PiL} we have
	\[
	\psi_\mu^{(p+1)}\left(\chi_\mu\left(-\frac1x\right)\right) =C+c_p\int_{x_0}^x(1+o(1))\frac{L(t)}{t} \dd t+d(1+o(1))L(x)
	\]
	where $c_p=(p+1)!(p+1)$, $x_0\geq0$ and $C, d$ are real constants.
	Setting $x=-1/\psi_\mu(-1/y)$ we have
	\begin{align*}
		\psi_\mu^{(p+1)}\left(-\frac{1}{y}\right) &=C+c_p\int_{x_0}^{\frac{-1}{\psi_\mu(-1/y)}}(1+o(1))\frac{L(t)}{t} \dd t+d(1+o(1))L\left(\frac{-1}{\psi_\mu(-1/y)}\right)  \end{align*}
	observe that since $\psi_\mu(-1/s)\sim -m_1(\mu)/s$, by the Uniform Convergence Theorem, we have $L\left({-1}/{\psi_\mu(-1/y)}\right) = L(y)(1+o(1))$.
	Further, substituting $t=-1/\psi_\mu(-1/s)$ in the integral above we obtain
	\begin{align*}
		\int_{\frac{-1}{\chi_\mu(-1/x_0)}}^{y}(1+o(1))\frac{L(-1/\psi_\mu(-1/s))}{-1/\psi_\mu(-1/s)} \frac{\psi_\mu^\prime(-1/s)}{\psi_\mu(-1/s)^2}\frac{1}{s^2}\dd s.
	\end{align*}
	Since $\psi_\mu(-1/s)\sim -m_1(\mu)/s$ and  $\psi_\mu^\prime(-1/s)\sim m_1(\mu)$, the integrand equals
	\[
	(1+o(1))\frac{L(s)}{s}.
	\] 
	Claim follows from another use of Theorem \ref{thm:PiL}.

	\textsc{Claim 6.} Condition \eqref{eq:L} holds.\\
	By \eqref{eq:Psipp1} and Theorem \ref{thm:taubPi} we obtain \eqref{eq:L}.
\end{proof}

\section{Proof of Theorem \ref{Introthm:1.4}}\label{sec:5}

This section is devoted to the proof of Theorem \ref{Introthm:1.4} announced in the Introduction. For reader's convenience we repeat the statement, see  Theorem \ref{thm:pat0} below.

First we determine roughly speaking how the tail of a measure changes when the $S$-transform is asymptotically multiplied by a constant $c>0$.

\begin{Corollary}\label{cor:first0}
	Assume that $\mu\in\mathcal{M}_+$ is such that
	\begin{align}\label{eq:newL}
		\bar{\mu}(x)\sim x^{-\alpha}L(x).
	\end{align}
	where $\alpha\in(0,1)$. 
	Let $c>0$, $\nu\in\mathcal{M}_+$ and assume that
	\[
	S_\nu\left(-\frac1x\right)\sim c\, S_\mu\left(-\frac1x\right).
	\]
	Then
	\[
	\bar{\nu}(x)\sim c^{-\alpha} \bar{\mu}(x).
	\]
\end{Corollary}
\begin{proof}
	For $M=L^{-1/\alpha}\in \mathcal{R}_0$ define $N_c := c^{-1}M^\#$. Observe that de Bruijn conjugate $N_c^\#$ to a slowly varying function $N_c$ is asymptotically equivalent to $c\,M=c\,L^{-1/\alpha}$, which is easily seen from the definition. In view of \eqref{eq:newL}, by Theorem \ref{thm:01} we have
	\[
	S_{\nu}\left(-\frac1x\right)\sim \left( \frac{\sin(\pi \alpha)}{\pi \alpha} \right)^{1/\alpha}\frac{x^{1-1/\alpha}}{ N_c(x^{1/\alpha})}.
	\] 
	Using the converse implication from Theorem \ref{thm:01} we obtain that
	\[
	\bar{\nu}(x)\sim x^{-\alpha} N_c^\#(x)^{-\alpha}\sim c^{-\alpha}x^{-\alpha} L(x).
	\]
\end{proof}

\begin{Corollary}\label{cor:firstA0}
	Assume that $\mu\in\mathcal{M}_+$ is such that $\bar{\mu}\in \mathcal{R}_0$. Let $c>0$, $\nu\in\mathcal{M}_+$ and assume that
	\begin{align}\label{eq:SKR}
		S_\nu\left(-\frac1x\right)\sim c\, S_\mu\left(-\frac1x\right).
	\end{align}
	Then
	\[
	\bar{\nu}(x)\sim \bar{\mu}(x).
	\]
\end{Corollary}
\begin{proof}
	Let us define  $f_\sigma(x)=(x-1)/S_\sigma(-1/x)$  for $\sigma=\mu,\nu$. By Theorem \ref{thm:a0} and Remark \ref{rem:KR} we have that $f_\mu\in K\mathcal{R}_\infty$. It is easy to verify that assumptions of Theorem \ref{thm:KR} $(i)$ are satisfied by $f_\mu^{\langle-1\rangle}$ (note that in Theorem \ref{thm:KR} as $f$ we take the function $f_\mu^{\langle-1\rangle}$) thus $f_\mu^{\langle-1\rangle}$ is slowly varying. Thus, by \eqref{eq:SKR} we have
	\[
	f_\nu^{\langle-1\rangle}(x)\sim f_\mu^{\langle-1\rangle}(c x)\sim f_\mu^{\langle-1\rangle}(x).
	\] 
	Applying once again Theorem \ref{thm:KR} for $f^{\langle-1\rangle}_\nu$ we obtain that $f_\nu\in K\mathcal{R}_\infty$. Another application of Remark \ref{rem:KR} together with the converse implication of Theorem \ref{thm:a0} yields the result.
\end{proof}

\begin{Corollary}\label{cor:firstA1}
	Assume that $\mu\in\mathcal{M}_+$ is such that
	\begin{align*}
		\bar{\mu}(x)\sim L(x)/x,
	\end{align*}
where  $\int_1^{+\infty} L(t)/t\,\dd t=+\infty$.
	Let $c>0$, $\nu\in\mathcal{M}_+$ and assume that
	\[
	S_\nu\left(-\frac1x\right)= c\, S_\mu\left(-\frac1x\right)+o(x^{-\eps})
	\]
	for some $\eps>0$.
	Then
	\[
	\bar{\nu}(x)\sim c^{-1} \bar{\mu}(x).
	\]
\end{Corollary}

\begin{proof}
	By Theorem \ref{thm:alpha1} we have
	\[
	\frac{1}{S_\nu\left(-1/x\right)} = \frac{1}{c\, S_\mu\left(-1/x\right)}+o(x^{-\eps})\in \Pi_M(c^{-1}),
	\]
	where $M(x)=L\big((x-1)/S_\mu(-1/x)\big)$. Since $L\in \mathcal{R}_0$, we have
	$M(x)\sim L\big((x-1)/S_\nu(-1/x)\big)$, which implies that
	\[
	\Pi_M(c^{-1}) = \Pi_{\widetilde{M}_c}(1),
	\]
	where $\widetilde{M}_c(x)=c^{-1}L\big((x-1)/S_\nu(-1/x)\big)$. By the converse implication of  Theorem \ref{thm:alpha1} we obtain the result.
\end{proof}

\begin{Corollary}\label{cor:first}
	Assume that $p\in\mathbb{N}$, $\mu\in\mathcal{M}_p$ with 
	\[
	\bar{\mu}(x)\sim x^{-\alpha}L(x).
	\]
	Let $c>0$ and $\nu\in\mathcal{M}_+$. 
	Moreover \begin{enumerate}
			\item for $\alpha\in[p,p+1)$, assume that
			\[
			S_\nu^{(p)}\left(-\frac1x\right)\sim c\, S_\mu^{(p)}\left(-\frac1x\right),
			\]
			\item for $\alpha=p+1$ suppose
			\[
			x\mapsto S_\nu^{(p)}\left(-\frac1x\right)-c\, S_\mu^{(p)}\left(-\frac1x\right)\in \Pi_L(0).
			\]
		\end{enumerate}
	In both cases we have
	\[
	\bar{\nu}(x)\sim c \left(\frac{m_1(\nu)}{m_1(\mu)}\right)^{\alpha+1} \bar{\mu}(x).
	\]
\end{Corollary}

\begin{proof}
	We consider two cases:
	\begin{description}
		\item[$\alpha\in[p,p+1)$]  By Remark \ref{rem:ap} we have $\int_1^{+\infty} L(t)/t\,\dd t<+\infty$ if $\alpha=p$. Thus, by Theorem \ref{thm:p} $(i)$ we obtain 
		\begin{align*}
			S_{\nu}^{(p)}\left(-\frac1x\right) &
			\sim 
			- c \frac{\Gamma(\alpha+1)\Gamma(p+1-\alpha)}{m_1(\mu)^{\alpha+1}} x^{p+1-\alpha}L(x) \\
			& =-\frac{\Gamma(\alpha+1)\Gamma(p+1-\alpha)}{m_1(\nu)^{\alpha+1}} x^{p+1-\alpha}  L_c(x),
		\end{align*}
		where we  denoted $L_c(x) = c\, (m_1(\nu)/m_1(\mu))^{\alpha+1}L(x)\in \mathcal{R}_0$.
		Again using Theorem \ref{thm:p} $(i)$ we obtain the assertion.
		
		\item[$\alpha=p+1$] In this case, Remark \ref{rem:ap} ensures that $\int_1^{+\infty} L(t)/t\,\dd t=+\infty$. Under the assumptions, Theorem \ref{thm:p} $(ii)$ implies 
		\begin{align*}
			S_{\nu}^{(p)}\left(-\frac1x\right) \in  \Pi_L\left(- c\frac{(p+1)!}{m_1(\mu)^{p+2}}\right)=\Pi_{L_{c}}\left(-\frac{(p+1)!}{m_1(\nu)^{p+2}}\right),
		\end{align*}
		where $L_{c}$ is defined as above. Using Theorem \ref{thm:p} $(ii)$ again yields the result.
	\end{description}
\end{proof}

\begin{thm}\label{thm:pat0}\  
		\begin{enumerate}[(i)]
		\item If $\mu\in\mathcal{M}_+$ and $\bar{\mu}(x)\sim L(\log(x))/\log(x)^\beta$ for $L\in\mathcal{R}_0$ and $\beta>0$, then for $t\geq 1$,
		\[
		\mu^{\boxtimes t}\big((x,+\infty)\big)\sim t^\beta\,\mu((x,+\infty).
		\]
		\item  Assume $\mu\in\mathcal{M}_+$  satisfies \eqref{Introeq:L} for $\alpha\in(0,1)$. Then for $t\geq 1$ we have $\mu^{\boxtimes t}\big((\cdot,+\infty)\big)\in \mathcal{R}_{-\alpha_t}$, where 
		\[
		\alpha_t=\frac{\alpha}{\alpha+t(1-\alpha)}.
		\]
		In particular, if $\bar{\mu}(x)\sim c/x^{\alpha}$ for some $\alpha\in(0,1)$ and $c>0$, then for $t\geq 1$ one has,
		\begin{align*}
			\mu^{\boxtimes t}\big((x,+\infty)\big)\sim \frac{c_{t,\alpha}}{x^{\alpha_t}},
		\end{align*}
		where 
		\[
		c_{t,\alpha} = \left(c\frac{\pi\alpha}{\sin(\pi\alpha)}\right)^{t/(\alpha+t(1-\alpha))} \frac{\sin(\pi\alpha_t)}{\pi\alpha_t}.
		\]
		\item If $\mu\in\mathcal{M}_+$ is such that $\bar{\mu}(x)\sim c/x$ with $c>0$, then for $t\geq 1$,
		\begin{align*}
			\mu^{\boxtimes t}\big((x,+\infty)\big)\sim c^{t-1}t \log(x)^{t-1}\mu\big((x,+\infty)\big).
		\end{align*}
		\item	Let $\alpha\geq1$ and assume $\mu\in\mathcal{M}_+$  satisfies \eqref{Introeq:L} and $m_1(\mu)<+\infty$.   Then for $t\geq 1$ we have 
		\begin{align*}
			\mu^{\boxtimes t}\big((x,+\infty)\big)\sim t\, m_1(\mu)^{\alpha(t-1)} \mu\big((x,+\infty)\big).
		\end{align*}
	\end{enumerate}
\end{thm}

\begin{proof}
	\begin{enumerate}[$(i)$]
		\item Recall that for $\mu\in\mathcal{M}_+$ we defined $f_\mu(x)=(x-1)/S_\mu(-1/x)$. Denote $f_t=f_{\mu^{\boxtimes t}}$. Since $S_{\mu^{\boxtimes t}}=S_\mu^t$, we have
		\begin{align}\label{eq:slow}
		f_{1}(x)=(x-1)^{1-1/t}f_{t}(x)^{1/t},
		\end{align}
		which is equivalent to
		\[
		f_t^{\langle-1\rangle}(x)=f_1^{\langle-1\rangle}\left(x^{1/t} \left(f_t^{\langle-1\rangle}(x)-1\right)^{1-1/t}\right).
		\]
		Note that since $\bar{\mu}\in\mathcal{R}_0$, by Theorem \ref{thm:a0} we have $f_1\in K\mathcal{R}_\infty$. This and \eqref{eq:slow} in turn imply that $f_t\in K\mathcal{R}_\infty$ for any $t\geq1$. Whence, by Theorem \ref{thm:KR} $(i)$ we obtain that $f_t^{\langle-1\rangle}\in\mathcal{R}_0$ and so 
		 \begin{align}\label{eq:slow2}
		\log\left(f_t^{\langle-1\rangle}(x)\right)=o(\log(x)).
		 \end{align}
		By Theorem \ref{thm:a0} we have $1/f_1^{\langle-1\rangle}(x)\sim g\left(\log(x)\right)$, where $g(x):=x^{-\beta}L(x)$. In view of \eqref{eq:slow2}, we have
		\begin{align*}
		\frac{1}{f_t^{\langle-1\rangle}(x)}&\sim  g\left(\log\left(x^{1/t} \left(f_t^{\langle-1\rangle}(x)-1\right)^{1-1/t}\right)\right)\sim g\left(\log(x^{1/t})\right)\\
		&=g\left(t^{-1} \log(x)\right)\sim t^\beta g(\log(x)),
		\end{align*}
		where we have used the fact that  $g\in\mathcal{R}_{-\beta}$.

		The converse implication of Theorem \ref{thm:a0} yields the result.
		\item This is a straightforward application of Theorem \ref{thm:01} with $L\equiv c>0$, for the case $L=c$ we use \eqref{eqn:4.4}.
		\item 	In view of Theorem \ref{thm:alpha1}, it is enough to show that 
		\[
		x\mapsto 1/S_{\mu^{\boxtimes t}}(-1/x)\in \Pi_{\widetilde{M}}(1),
		\]
		where $\widetilde{M}(x)=M\big((x-1)/S_{\mu^{\boxtimes t}}(-1/x)\big)$ and $M(x)=c^t\, t \log(x)^{t-1}$. 	
		
		By direct implication of 
		Theorem \ref{thm:alpha1} we have $x\mapsto 1/S_\mu(-1/x)\in \Pi_{L_1}(1)$, where $L_1(x)=c$. In particular, this implies that $1/S_\mu(-1/x)\sim c \log(x)$.
		
		For $t\geq1$, by L' Hospital's rule we have 
			\[
			\lim_{z\to1} \frac{z^t-1}{t(z-1)}=1.
			\]
		Setting $z=S_\mu(-1/(\lambda x)) / S_\mu(-1/x)$, after simple rearrangements, we obtain
		\begin{align*}
		\frac{1}{S_\mu(-1/(\lambda x))^t} - 		\frac{1}{S_\mu(-1/x)^t} &\sim t \left(\frac{1}{S_\mu(-1/(\lambda x))} - 		\frac{1}{S_\mu(-1/x)}\right) \frac{1}{S_\mu(-1/x)^{t-1}}\\
		&\sim t \,\big(c\log(\lambda)\big) \,\big(c\log(x)\big)^{t-1} = \log(\lambda) M(x),
		\end{align*}
	where the latter asymptotic equivalence follows from previous observations. Thus, we see that  $x\mapsto 1/S_{\mu^{\boxtimes t}}(-1/x)\in \Pi_{M}(1)$.
		
		It is left to show that 
		$M(x)\sim \widetilde{M}(x)$. 
		We have
		\begin{align*}
			\widetilde{M}(x)&=M\left(\frac{x-1}{S_\mu^t(-1/x)}\right)\sim M\left(x \log(x)^{t}\right)\\
			&= c^t t \log\big(x\log(x)^t\big)^{t-1}\sim c^t t \log\left(x\right)^{t-1}=M(x).
		\end{align*}
		\item Let $p\in\mathbb{N}$ be such that $\mu\in\mathcal{M}_p$. We consider two cases (see Remark \ref{rem:ap}):
		\begin{description}
			\item[$\alpha\in[p,p+1)$]. By Theorem \ref{thm:p} $(i)$ we see that $S_\mu^{(k)}(0^-)$ is finite for all $k<p$, $S_\mu(0^-)=1/m_1(\mu)$, while $S_\mu^{(p)}(-1/x)\to+\infty$ as $x\to+\infty$.  Thus, we obtain 
			\begin{align*}
				S_{\mu^{\boxtimes t}}^{(p)}\left(-\frac1x\right) &= \frac{\dd^p S_\mu^t}{\dd z^p}\left(-\frac1x\right) \sim t \,S_\mu^{t-1}\left(-\frac1x\right) S_\mu^{(p)}\left(-\frac1x\right) \\
				& \sim 
				t \frac{1}{m_1(\mu)^{t-1}}S_\mu^{(p)}\left(-\frac1x\right).
			\end{align*}
			Since $m_1(\mu^{\boxtimes t})=m_1(\mu)^t$, we have
			\[
			\frac{t}{m_1(\mu)^{t-1}}\left(\frac{m_1(\mu^{\boxtimes t})}{m_1(\mu)}\right)^{\alpha+1}=t\, m_1(\mu)^{(t-1)\alpha}
			\]
			and the result follows from Corollary \ref{cor:first}.

			\item[$\alpha=p+1$] By Theorem \ref{thm:p} $(ii)$  we get
			\[
			x\mapsto S_{\mu}^{(p)}\left(-\frac1x\right) \in \Pi_L\left(-\frac{(p+1)!}{m_1(\mu)^{p+2}}\right).
			\]
			Similarly as before, the only term that contributes to asymptotics of $S_{\mu^{\boxtimes t}}^{(p)}(z)$ is $t S_\mu^{t-1}(z) S_\mu^{(p)}(z)$. Indeed, we have $S_\mu^{(n)}(z)=S_\mu^{(n)}(0^-)+o(z^\varepsilon)$ (see the argument after \eqref{eq:chipnew}) for any $\varepsilon\in(0,1)$. Thus,
			\begin{align*}
				S_{\mu^{\boxtimes t}}^{(p)}\left(-\frac1x\right)&=t S_\mu^{t-1}\left(-\frac1x\right) S_\mu^{(p)}\left(-\frac1x\right) + c + o(x^{-\eps}) \\
				&	=  t \frac{1}{m_1(\mu)^{n-1}} S_\mu^{(p)}\left(-\frac1x\right) + c + o(x^{-\eps})
			\end{align*}
			for $c\in\R$.
			Thus, the result follows from Corollary \ref{cor:first}.
		\end{description}  
	\end{enumerate}
\end{proof}

\section{$\boxtimes$-infinitely divisible laws}\label{sec:6}

In this section we apply the machinery developed in Section \ref{sec:4} to study the relation between tails of L\'evy measure and tails of $\boxtimes$ infinitely divisible measure. The free additive case was studied in \cite{CCH18}.

A measure $\mu\in\mathcal{M}_+$ is said to be $\boxtimes$-infinitely divisible ($\boxtimes$-ID) if, for every $n\in\mathbb{N}$, there exists a measure $\nu_n\in\mathcal{M}_+$ such that $\mu = \nu_n^{\boxtimes n}$. The set of $\boxtimes$-ID distributions was characterized in terms of $S$-transforms in \cite[Theorem 6.13]{BV93}.
\begin{thm}
	A measure $\mu\in\mathcal{M}_+$ is $\boxtimes$-ID if and only if there exists a finite positive Borel measure $\sigma$ on the compact space $[0,+\infty]$ and a real number $\gamma$ such that $S_\mu(z)=\exp(v(z))$, where $v$ is defined by
	\begin{align}\label{eq:idd}\begin{split}
			v\left(\frac{z}{1-z}\right)&=\gamma+\int_{[0,+\infty]}\frac{1+tz}{z-t}\sigma(\dd t) \\
			&= \gamma + \sigma(\{0\})\frac1z- \sigma(\{+\infty\} )z+\int_{(0,+\infty)}\frac{1+tz}{z-t}\sigma(\dd t).
	\end{split}\end{align}
\end{thm}
Let $\mu_\boxtimes^{\gamma,\sigma}$ be the $\boxtimes$-ID measure determined by \eqref{eq:idd}. Measure $\sigma$ is called the L\'evy measure for $\mu_\boxtimes^{\gamma,\sigma}$ . We will describe asymptotics of right tail of $\mu_\boxtimes^{\gamma,\sigma}$ when its L\'evy measure has regularly varying left tail. 
A corollary to results that will be presented later is the following.
\begin{Corollary}
	Assume that $\sigma$ has regularly varying left (resp. right) tail with index $-\alpha\leq0$. If $\alpha=1$, assume additionally that limit of $x\sigma\big([0,x^{-1})\big)$ (resp. $x\sigma\big((x,\infty)\big)$) exists as $x\to+\infty$. Then $\mu_\boxtimes^{\gamma,\sigma}$ has regularly varying right (resp. left) tail.
\end{Corollary}
We conjecture that an additional condition for $\alpha=1$ is not necessary and that the converse is also true: regular variation of tails of $\mu_\boxtimes^{\gamma,\sigma}$ implies the regular variation of tails of $\sigma$.

We will now proceed to the description of the right tail $\mu_\boxtimes^{\gamma,\sigma}$ in terms of the asymptotics of $\sigma\big([0,x)\big)\sim x^\alpha L(1/x)$ as $x\to 0^+$, $L\in\mathcal{R}_0$. The case $\alpha\in[0,1)$ is treated in Theorem \ref{thm:IDD1}, the case $\alpha=1$ (which is most complex) is dealt with in Theorem \ref{thm:IDD2}, while $\alpha\in[p,p+1]$ for $p\in\mathbb{N}$ can be found in Theorem \ref{thm:IDD3}. Recall that by $\hat{\sigma}$ we denote the pushforward of $\sigma$ under the mapping $x\mapsto 1/x$.
\begin{thm}\label{thm:IDD1}
	Let $\alpha\in[0,1)$ and let $L\in\mathcal{R}_0$. Assume that $\sigma$ and $v$ are related by \eqref{eq:idd}. The following two conditions are equivalent:
	\begin{align}
		\sigma\big([0,x)\big)&\sim x^\alpha L(1/x)\qquad\mbox{ as }\qquad x\to0^+, \label{eq:v1}\\
		v(-1/x)&\sim - \frac{\pi\alpha}{\sin(\pi\alpha)} x^{1-\alpha}L(x)\qquad\mbox{ as }\qquad x\to+\infty\label{eq:v2}.
	\end{align}
	Each of these equivalent conditions imply that for all $\gamma\in\R$,
	\begin{align}\label{eq:v3}
		\mu_\boxtimes^{\gamma,\sigma}\big((x,+\infty)\big)&\sim \left( \frac{\pi\alpha}{\sin(\pi\alpha)} \right)^{1/(1-\alpha)}   \frac{1}{\log(x) ^{1/(1-\alpha)} L_\alpha^\#(\log(x)^{1/(1-\alpha)})}\quad\mbox{as}\quad x\to+\infty,
	\end{align}
	where $L_\alpha:=L^{1/(1-\alpha)}\in\mathcal{R}_0$.
\end{thm}
\begin{proof}
	From \eqref{eq:v1} we obtain $\hat{\sigma}\big((x,+\infty]\big) \sim x^{-\alpha} L(x)$. Define 
	\[
	u(x):=x\int_{[0,+\infty)} \frac{1}{1+x t}\sigma(\dd t) = x\left(\sigma(\{0\})+ \int_{(0,+\infty)} \frac{t}{t+x}\hat{\sigma}(\dd t)\right).
	\]
	By Lebesgue's Dominated Convergence Theorem we have $u(x)/x \to \sigma(\{0\})$ as $x\to+\infty$. If $\sigma(\{0\})=0$ and \eqref{eq:v1} holds, then a more precise behavior of $u$ is available. Indeed, arguments similar to those used in the proof of Theorem \ref{thm:01} imply that in such case \eqref{eq:v1} is equivalent to 
	\[
	u(x)\sim \frac{\pi\alpha}{\sin(\pi\alpha)}x^{1-\alpha}L(x).
	\]
	Clearly, the above condition is true also in the case of $\sigma(\{0\})=\lim_{x\to0^+} x^0 L(1/x)>0$. 
	
	Direct calculations show that
	\[
	v\left(-\frac{1}{x+1}\right)=-u(x) + \gamma+\sigma(\{+\infty\}) \frac1x + \int_{(0,+\infty)} \frac{1}{t+x}\hat{\sigma}(\dd t) = -u(x)+\gamma+o(1).
	\]
	Clearly, we have $v(-1/(1+x))\sim -u(x)$ and the regular variation implies that $v(-1/(1+x))\sim v(-1/x)$.
	Thus, we have proved equivalence between \eqref{eq:v1} and \eqref{eq:v2}. 
	
	Since $S_{\mu_\boxtimes^{\gamma,\sigma}}(-1/x)=\exp(v(-1/x))$, under \eqref{eq:v2}, Corollary \ref{cor:expv} implies that 
	\[
	\mu_\boxtimes^{\gamma,\sigma}\big((x,+\infty)\big)\sim 1/ s^{\langle-1\rangle}(\log(x)),
	\]	
	where $s(t):=-v(-1/t)=\frac{\pi\alpha}{\sin(\pi\alpha)}t^{1-\alpha}L(t)$.
	Let $a=1-\alpha$ and recall that an asymptotic inverse of $t\mapsto t^{1-\alpha}L(t)=t^a L_\alpha^a(t)$ is $x\mapsto x^{1/a}L_\alpha^\#(x^{1/a}) $. Thus, 
	\[
	s^{\langle-1\rangle}(x)\sim \left(\frac{\sin(\pi\alpha)}{\pi\alpha}\right)^{1/(1-\alpha)} x^{1/(1-\alpha)} L_\alpha^\#(x^{1/(1-\alpha)})
	\]
	and the result follows.
\end{proof}

In the result below we show that in the case $\alpha=1$, the right tail of $\mu_\boxtimes^{\gamma,\sigma}$ is regularly varying with index depending on a limit at $+\infty$ of $L(x)\sim x\,\sigma\big([0,x^{-1})\big)$. We refrained from exhibiting {the slowly varying function related with right tail of $\mu_\boxtimes^{\gamma,\sigma}$} as it depends on $\sigma$ in a {quite} complicated way.  Nevertheless, it is hidden in the proof of Theorem \ref{thm:IDD2}. We do not know whether a similar result {holds when} $L$ does not have a limit at infinity.
\begin{thm}\label{thm:IDD2} 	Let $L\in\mathcal{R}_0$. 
	\begin{enumerate}[(i)]
		\item The following two conditions are equivalent:
		\begin{align}
			\sigma\big([0,&x)\big)\sim x L(1/x)\qquad\mbox{ as }\qquad x\to0^+, \label{eq:vv1}\\
			x\mapsto &v(-1/x)\in \Pi_L(-1)\label{eq:vv2}.
		\end{align}
	\item 	If $d:=\lim_{x\to+\infty}L(x)$ exists, then each of these equivalent conditions implies that
	\begin{align}\label{eq:Scases}
		x\mapsto \frac{1}{S_{\mu_\boxtimes^{\gamma,\sigma}}(-1/x)}\in\begin{cases}
			K\mathcal{R}_\infty, & \mbox{ if }d=+\infty, \\
			\mathcal{R}_d, & \mbox{ if }d\in(0,+\infty),\\
			\Pi_M(1), & \mbox{ if }d=0,
		\end{cases} 
	\end{align}
	where $M(x):=L(x)/S_{\mu_\boxtimes^{\gamma,\sigma}}(-1/x)$.
	\item If $d=0$ assume additionally that $\int_{(0,+\infty)}t^{-1}\sigma(\dd t)=+\infty$. 
	Then, \eqref{eq:Scases} implies that 
	\[
	x\mapsto\mu_\boxtimes^{\gamma,\sigma}\big((x,+\infty)\big)\in \mathcal{R}_{-1/(1+d)}.
	\]
	\end{enumerate}
\end{thm}

The proof of Theorem \ref{thm:IDD2} is based on the following result.
\begin{lemma}\label{lem:expPi}
	Assume that $L\in\mathcal{R}_0$ and $f\in\Pi_L(c)$, $c>0$. Let $g(x)=\exp(f(x))$.  Assume that $L(x)$ has a limit as $x\to+\infty$ and define $d=\lim_{x\to+\infty}L(x)$. If $d=+\infty$ assume additionally that $g$ is nondecreasing.
	We have 
	\[
	g \in \begin{cases}
		K\mathcal{R}_{\infty}, & \mbox{ if }d=+\infty,\\
		\mathcal{R}_{cd}, & \mbox{ if }d\in(0,+\infty),\\
		\Pi_{g\cdot L}(c), & \mbox{ if }d=0.
	\end{cases}
	\]
\end{lemma}
\begin{proof}
	We start with $d=+\infty$. In view of Theorem \ref{thm:KR} $(ii)$ it is enough to show that $g\in\mathcal{R}_\infty$, that is,
	$g(\lambda x)/g(x)\to+\infty$ as $x\to+\infty$ for all $\lambda>0$. We have as $x\to+\infty$,
	\[
	\log\left(\frac{g(\lambda x)}{g(x)}\right)=\frac{f(\lambda x)-f(x)}{L(x)}L(x)\sim c\log(\lambda)L(x)\to+\infty.
	\]

	Assume that $d\in(0,+\infty)$. We have $L(x)=d+o(1)$. By Theorem \ref{thm:PiL} we have for $C, D\in\R$ and $x_0\geq0$,
	\begin{align*}
		g(x)&=e^{f(x)}=\exp\left(C+c \int_{x_0}^x (1+o(1))\frac{L(t)}{t}\dd t + D(1+o(1))L(x)\right)\\
		&=\exp\left(c\,d\log(x)+\widetilde{C}+o(1)+ c_1 \int_{x_0}^x \frac{o(1)}{t}\dd t\right)= x^{c d}e^{\widetilde{C}+o(1)+ c_1 \int_{x_0}^x \frac{o(1)}{t}\dd t},
	\end{align*}
	where we have substituted $L(x)=d+o(1)$ in the third equality above. By \cite[Theorem 1.3.1]{BGT89} function $\ell(x):=\exp(C+o(1)+\int_1^x \frac{o(1)}{t}\dd t)$ is slowly varying. 
	
	If $d=0$, then 
	\[
	\frac{g(\lambda x)-g(x)}{g(x)L(x)}=\frac{e^{f(\lambda x)}-e^{f(x)}}{e^{f(x)}L(x)} = \frac{e^{\frac{f(\lambda x)-f(\lambda)}{L(x)}L(x)}-1}{L(x)}\to c\log(\lambda),
	\]
	which ends the proof.
\end{proof}

\begin{proof}[Proof of Theorem~\ref{thm:IDD2}]
	We first show the equivalence between \eqref{eq:vv1} and \eqref{eq:vv2}. We proceed as in the proof of Theorem \ref{thm:IDD1}. With $u(x)=x\int_{[0,+\infty)} \frac{t}{t+x}\hat{\sigma}(\dd t)$ we have 
	\[
	v\left(-\frac{1}{x}\right)=-u(x-1) + \gamma+\sigma(\{+\infty\}) \frac1{x-1} + \int_{(0,+\infty)} \frac{1}{t+x-1}\hat{\sigma}(\dd t) = -u(x-1)+\gamma+o(x^{-\eps})
	\] 
	for all $\eps\in(0,1)$. By Theorem \ref{thm:taubA1} we have the equivalence of \eqref{eq:vv1} and $u\in\Pi_L(1)$. Since $u(x)-u(x-1)=o(x^{-\eps})$, by Lemma \ref{lem:simpleT}, we have $x\mapsto u(x-1)\in\Pi_L(1)$. Thus, again using  Lemma \ref{lem:simpleT} we obtain \eqref{eq:vv2} and prove the first part of the assertion.
	
	For $(ii)$, recall that 
	\[
	\frac{1}{S_{\mu_\boxtimes^{\gamma,\sigma}}(-1/x)}= e^{-v(-1/x)}.
	\]
	Under \eqref{eq:vv2} we have $x\mapsto -v(-1/x)\in \Pi_L(1)$ and so \eqref{eq:Scases} follows from Lemma \ref{lem:expPi}.
	
	Point $(iii)$ for $d=+\infty$ follows from Theorem \ref{thm:a0} and for $d\in(0,+\infty)$ from Theorem \ref{thm:01}. The case of $d=0$ requires more work. 
	
Since $\int_{(0,+\infty)}t^{-1}\sigma(\dd t)=+\infty$, we have $v(0^-)=-\infty$. Define a function $\widetilde{L}\in\mathcal{R}_0$ by (recall that $(x-1)/S_\mu(-1/x)=-1/\chi_\mu(-1/x)$ is monotonic and continuous)
	\begin{align}\label{eq:Ltilde}
		\widetilde{L}\left(\frac{x-1}{S_\mu(-1/x)}\right )=	\frac{L(x)}{S_\mu(-1/x)}.
	\end{align}
	Thus, \eqref{eq:Scases} for $d=0$ implies that
	\[
	x\mapsto \frac{1}{S_{\mu_\boxtimes^{\gamma,\sigma}}(-1/x)} \in \Pi_{\widetilde{M}}(1),
	\]
	where $\widetilde{M}(x):=\widetilde{L}\left((x-1)/S_\mu(-1/x)\right )$. 
	
In view of Theorem \ref{thm:alpha1}, if we show that
	\begin{align}\label{eq:toboproved}
		\int_1^{+\infty} \widetilde{L}(t)/t\,\dd t=+\infty,
	\end{align} then we will have
	\[
	\mu_\boxtimes^{\gamma,\sigma}\big((x,+\infty)\big) \sim \widetilde{L}(x)/x.
	\]
	Denote $f_\mu(x)=(x-1)/S_\mu(-1/x)$ and observe that \eqref{eq:Ltilde} gives
	\[
	\frac{\widetilde{L}(f_\mu(x))}{f_\mu(x)} f_\mu^\prime(x)=\frac{L(x)}{x-1} f_\mu^{\prime}(x).
	\]
	Thus, we see that \eqref{eq:toboproved} holds true if 
	\[
	\int_1^x \frac{L(t)}{t-1} f_\mu^{\prime}(t)\dd t\to +\infty\qquad\mbox{as}\qquad x\to+\infty.
	\]
	But $\frac{L(t)}{t-1} f_\mu^{\prime}(t) \sim \left( {1}/{S_{\mu_\boxtimes^{\gamma,\sigma}}(-1/t)}\right)^\prime$ and ${1}/{S_{\mu_\boxtimes^{\gamma,\sigma}}(-1/x)}\to+\infty$ by the fact that $v(0^-)=-\infty$.
\end{proof}

We extend a definition of $\cM_+$ to finite measures on $[0,+\infty]$. Let $\mathcal{P}_+$ denote the set of finite Borel measures on $\overline{\R}_+$ and define for $p\in\mathbb{N}\cup\{0\}$,
\[
\widetilde{\mathcal{M}}_{p}:=\{\sigma\in\mathcal{P}_+\colon m_{p}(\sigma)<+\infty\mbox{ and }m_{p+1}(\sigma)=+\infty \}.
\]
It is easy to see that we have $v(0^-)=+\infty$ if and only if $\int_{[0,+\infty]}t^{-1}\sigma(\dd t)=m_1(\hat{\sigma})=+\infty$. Thus, by Proposition \ref{thm:useful} $(iii)$ we have 
\[
\mu_\boxtimes^{\gamma,\sigma}\in \mathcal{M}_0\qquad\mbox{if and only if}\qquad\hat{\sigma}\in\widetilde{\mathcal{M}}_{0}.
\] 
\begin{lemma}
	Let $p\in\mathbb{N}$. We have 
	\[
	\mu_\boxtimes^{\gamma,\sigma}\in \mathcal{M}_p\qquad\mbox{if and only if}\qquad \hat{\sigma}\in\widetilde{\mathcal{M}}_p.
	\]
\end{lemma}
\begin{proof} We will show that condition $\hat{\sigma}\in\widetilde{\mathcal{M}}_p$ implies $\mu_\boxtimes^{\gamma,\sigma}\in \mathcal{M}_p$. The reverse implication works along the same lines.
	
	In view of Lemma \ref{lem:ConvergenceS}, we know that condition $m_p(\mu_\boxtimes^{\gamma,\sigma})<+\infty$ is equivalent to series expansion in $0^-$ of order $p-1$ of 	$S_{\mu_\boxtimes^{\gamma,\sigma}}(z)=\exp(v(z))$. This will be established upon showing series expansion in $0^-$ of order $p-1$ of $v(z/(1-z))$. For $s=t^{-1}\in(0,+\infty)$ we have for any $n\in\mathbb{N}\cup\{0\}$ 
	\[
	-\frac{1+tz}{z-t} = s+\sum_{k=1}^n (s^{k+1}+s^{k-1})z^k+\frac{(1+s^2)s^n z^{n+1}}{1-sz},
	\]
	which implies that for $n=p-1$ (see \eqref{eq:idd})
	\begin{align*}
		v\left(\frac{z}{1-z}\right)= \gamma - m_1(\hat{\sigma})-\sum_{k=0}^{p-1} (m_{k+1}(\hat{\sigma})+m_{k-1}(\hat{\sigma}))z^k + z^{p-1}h(z)
	\end{align*}
where $h(z)=-\int_{[0,+\infty]} \frac{(1+s^2)s^{p-1} z}{1-sz} \hat{\sigma}(\dd s)$. 
If $m_p(\hat{\sigma})<+\infty$, then for $z\in(-1,0)$ we get
\[
\left|\frac{(1+s^2)s^{p-1} z}{1-sz}\right|\leq s^{p-1}+s^p 
\]
 and by Lebesgue's Dominated Convergence Theorem we obtain $h(z)=o(1)$ as $z\to0^-$. If $m_{p+1}(\hat{\sigma})=+\infty$, then clearly $v(z/(1-z))$ does not have series expansion of order $p$ and the same applies to $\exp(v(z))$. The result follows.
\end{proof}

If $m_{-1}(\sigma)<\infty$ (and so $\sigma(\{0\})=0$), then by Proposition \ref{thm:useful} $(iii)$ we have
\[
m_1(	\mu_\boxtimes^{\gamma,\sigma})=1/S_\mu^{(p)}(0^-)=\exp(-v(0^-))=\exp\left( -\gamma+m_{-1}(\sigma) \right).
\]

\begin{thm}\label{thm:IDD3}
	Let $p\in \mathbb{N}$, $\alpha\in[p,p+1]$ and $L\in\mathcal{R}_0$ be such that
	\begin{enumerate}
		\item[(a)] $\int_1^\infty L(t)/t \,\dd t<\infty$ if $\alpha=p$,
		\item[(b)] $\int_1^\infty L(t)/t \,\dd t=\infty$ if $\alpha=p+1$.
	\end{enumerate}
	The following two conditions are equivalent:
	\begin{align}
		&\sigma\big((0,x]\big)\sim x^\alpha L(1/x)\qquad\mbox{ as }\qquad x\to0^+,\label{eq:IDp1}\\
		\label{eq:IDp2}&\begin{cases}
			v^{(p)}(-1/x) \sim - \Gamma(\alpha+1)\Gamma(p+1-\alpha)x^{p+1-\alpha}L(x), & \mbox{if }\alpha\in[p,p+1), \\
			v^{(p)}(-1/x)  \in \Pi_L(-(p+1)!), & \mbox{if }\alpha=p+1
		\end{cases}
	\end{align}	
	Each of these equivalent conditions imply that
	\begin{align}\label{eq:iddP}
		\mu_\boxtimes^{\gamma,\sigma}\big((x,\infty)\big) \sim m_1(\mu_\boxtimes^{\gamma,\sigma})^{\alpha+1}\frac{L(x)}{x^\alpha}.
	\end{align}
\end{thm}
\begin{proof}
	First we show that \eqref{eq:IDp1} implies \eqref{eq:IDp2} for $\alpha\in[p,p+1)$. The other direction and the case $\alpha=p+1$  is proven similarly.
	
	Since $\frac{1+t z}{z-t} = -(z+z^{-1})\frac{s z}{1-sz}-z$ for $s=t^{-1}\in[0,\infty)$, by the definition of $v$ we obtain
	\begin{align}\label{eq:vnew}
		v\left(\frac{z}{1-z}\right)=\gamma+\sigma(\{0\})\frac{1}{z}-z\,\sigma\big((0,\infty]\big)-(z+z^{-1})\psi_{\hat{\sigma}}(z),
	\end{align}
	where $\psi_{\hat{\sigma}}(z)=\int_{[0,+\infty]}zt/(1-zt)\hat{\sigma}(\dd t)$ is the moment transform of $\hat{\sigma}\in\mathcal{P}_+$. Under \eqref{eq:IDp1} we have $\sigma(\{0\})=0$. 
	We will need the following easy facts, which we state without proofs.
	\begin{enumerate}[(a)]
		\item By Theorem \ref{thm:TaubMom} we have 
		\begin{align}\label{eq:UU1}
			\int_{[0,x]} t^{p+1} \hat{\sigma}(\dd t) \sim \frac{\alpha}{p+1-\alpha}x^{p+1-\alpha}L(x).
		\end{align}
		\item Define $U(x)$ by the l.h.s. of \eqref{eq:UU1}. By Theorem \ref{thm:TaubSt} we have 
		\begin{align*}
			\int_{[0,+\infty)}\frac{\dd U(s)}{(s+x)^{p+1}}
			&\sim  \frac{\Gamma(\alpha+1)\Gamma(p+1-\alpha)}{p!} x^{-\alpha}L(x).
		\end{align*}
		\item By direct calculation we obtain
		\[
		\frac{\dd^p}{\dd z^p} \frac{\psi_{\hat{\sigma}}(z)}{z}  = p!  \int_{[0,+\infty)} \frac{s^{p+1}}{(1-s z)^{p+1}}\hat{\sigma}(\dd s) = p!  \int_{[0,+\infty)} \frac{\dd U(s)}{(1-s z)^{p+1}}.
		\]
		\item Combining (b) and (c) above, we get
		\begin{align*}
			\left. \frac{\dd^p}{\dd z^p} \frac{\psi_{\hat{\sigma}}(z)}{z}\right|_{z=-1/(x+1)} & = (x+1)^{p+1} p!\int_{[0,+\infty)}\frac{\dd U(s)}{(s+(x+1))^{p+1}}\\
			&\sim \Gamma(\alpha+1)\Gamma(p+1-\alpha) x^{p+1-\alpha}L(x),
		\end{align*}
		where we have used the fact that for slowly varying functions we have $L(x)\sim L(x+1)$.
		\item By previous point we see that $\frac{\dd^p}{\dd z^p} (\psi_{\hat{\sigma}}(z)/z)\to+\infty$, while $\frac{\dd^k}{\dd z^k} (\psi_{\hat{\sigma}}(z)/z)$ has a finite limit as $z\to0^-$ for all $k<p$. Hence, by  \eqref{eq:vnew}, we obtain 
		\[
		\frac{\dd^p}{\dd z^p} \frac{\psi_{\hat{\sigma}}(z)}{z} \sim  \frac{\dd^p}{\dd z^p}\left((z+z^{-1})\psi_{\hat{\sigma}}(z)\right)  \sim  -\frac{\dd^p}{\dd z^p}  v\left(\frac{z}{1-z}\right) 
		\]
		\item As a consequence of (d) and (e) we see that  $v\left(\frac{z}{1-z}\right)$ has a finite limit as $z\to0^-$. Thus,
		\[
		\frac{\dd^p}{\dd z^p} v\left(\frac{z}{1-z}\right) \sim v^{(p)}\left(\frac{z}{1-z}\right).
		\]
		\item By (f) and (e) we have
		\[
		v^{(p)}\left(-\frac1x\right) \stackrel{\mbox{(f)}}{\sim} \left. \frac{\dd^p}{\dd z^p} v\left(\frac{z}{1-z}\right) \right|_{z=-1/(x+1)}\stackrel{\mbox{(e)}}\sim -\left.\frac{\dd^p}{\dd z^p} \frac{\psi_{\hat{\sigma}}(z)}{z} \right|_{z=-1/(x+1)}
		\]
		and so \eqref{eq:IDp2} follows from (d).
	\end{enumerate}
	
	Now we will prove the implication from \eqref{eq:IDp2} to \eqref{eq:iddP}. Again, we present only  the case $\alpha\in[p,p+1)$.  
	\begin{enumerate}[(A)]
		\item Under \eqref{eq:IDp2} we have
		\[
		S_{\mu_\boxtimes^{\gamma,\sigma}}^{(p)}(z)\sim e^{v(z)} v^{(p)}(z)\sim S_{\mu_\boxtimes^{\gamma,\sigma}}(0^-) v^{(p)}(z).
		\]
		\item By (A) and \eqref{eq:IDp2} we get
		\[
		S_{\mu_\boxtimes^{\gamma,\sigma}}^{(p)}\left(-\frac1x\right) \sim - \frac{\Gamma(\alpha+1)\Gamma(p+1-\alpha)x^{p+1-\alpha}}{m_1(\mu_\boxtimes^{\gamma,\sigma})^{1+\alpha}}\tilde{L}(x),
		\]
		where $\tilde{L}(x)=m_1(\mu_\boxtimes^{\gamma,\sigma})^{1+\alpha} L(x)$. 
		\item By the converse implication of Theorem \ref{thm:p} we obtain \eqref{eq:iddP}.
	\end{enumerate}
\end{proof}

The above results imply that the left tail of $\mu_\boxtimes^{\gamma,\sigma}$  can be found under regular variation of the right tail of $\sigma$. The argument is based on the following easy result.  
\begin{lemma}
	\[
	\hat{\mu}_\boxtimes^{\gamma,\sigma} = 		\mu_\boxtimes^{-\gamma,\hat{\sigma}}.
	\]
\end{lemma}
\begin{proof}
	By Proposition \ref{thm:useful} $(v)$, we have 
	\[
	S_{\hat{\mu}_\boxtimes^{\gamma,\sigma}}(z)= \exp(-v(-1-z)) =: \exp(\hat{v}(z)),
	\]
	where
	\begin{align*}
		\hat{v}\left(\frac{z}{1-z}\right)&=-v\left(-1-\frac{z}{1-z}\right)=-v\left(\frac{z^{-1}}{1-z^{-1}}\right) = -\gamma - \int_{[0,+\infty]} \frac{1+z^{-1}t}{z^{-1}-t}\sigma(\dd t) \\
		&= -\gamma 	+ \int_{[0,+\infty]} \frac{1+zt}{z-t}\hat{\sigma}(\dd t).
	\end{align*}
\end{proof}

\begin{example}\label{example:1.66}
	Let us consider a measure $\sigma_\alpha\in\mathcal{P}_+$ defined by its cumulative distribution function $\sigma_\alpha([0,x))=c\min\{x^\alpha,d^\alpha\}$, $x\geq0$ for $\alpha\geq0$, $d, c>0$. Then,
	\[
		\mu_\boxtimes^{\gamma,\sigma_\alpha}\big((x,\infty)\big)\sim
		\begin{cases}
 \left( \frac{\pi\alpha}{\sin(\pi\alpha)} \right)^{1/(1-\alpha)}  c^{1/(1-\alpha)} \log(x) ^{-1/(1-\alpha)}  & \mbox{for }\alpha\in[0,1), \\
\frac{\pi/(1+c)}{\sin(\pi/(1+c))}d^{c/(1+c)}e^{-\gamma/(1+c)}x^{-1/(1+c)} & \mbox{for }\alpha=1,\\
e^{-(1+\alpha)\gamma +\alpha d^{\alpha+1}} c \,x^{-\alpha} & \mbox{for }\alpha>1.
		\end{cases}
	\]
	
The case $\alpha\in[0,1)$ is a straightforward corollary from Theorem \ref{thm:IDD1} with $L=c$ and $L_\alpha=c^{1/(1-\alpha)}=1/L_\alpha^\#$.

The case $\alpha=1$ follows from a tedious calculations leading to
	\begin{align*}
	S_{	\mu_\boxtimes^{\gamma,\sigma_\alpha}}(-1/x)&=(1+d(x+1))^{-c(1+(x+1)^{-2})}e^{cd/(x+1)+\gamma}\\
	&\sim d^{-c}e^{\gamma} x^{-c} = C^{-1/\Lambda} \left(\frac{ \sin(\pi\Lambda)}{\pi\Lambda}\right)^{1/\Lambda} x^{1-1/\Lambda},
	\end{align*}
where $\Lambda=1/(1+c)$ and $C=\frac{\pi/(1+c)}{\sin(\pi/(1+c))}d^{c/(1+c)}e^{-\gamma/(1+c)}$. Thus, the result follows from the Remark after Theorem \ref{thm:01}.

In case $\alpha>1$, we have $m_1(\mu_\boxtimes^{\gamma,\sigma_\alpha})=\exp(-\gamma+m_{-1}(\sigma))=\exp(-\gamma + \alpha d^{\alpha+1}/(\alpha+1))$ and the result follows from Theorem \ref{thm:IDD3}.
\end{example}

\section{Miscellaneous remarks}\label{sec:7}
In this section we provide some consequences of results from Section \ref{sec:4}, which are not related to the phase transitions from previous sections, but still are of independent interest.
\subsection{Free Breiman's Lemma}

A special case of classical Breiman's Lemma was proved in \cite{Br65}. Its refinements can be found in \cite{Br07}. We present the most popular statement below.
\begin{proposition}[Breiman's Lemma]\label{Introlem:BrC}
	Let $L\in \mathcal{R}_0$. 	If $\mu,\nu\in\mathcal{M}_+$ are such that
	\[
	\mu\big((x,+\infty))\sim x^{-\alpha}L(x)\qquad\mbox{ and }\qquad m_{\alpha+\eps}(\nu)<+\infty,
	\]
	for $\alpha\geq0$ and $\eps>0$,
	then
	\[
	(\mu\circledast \nu)\big((x,+\infty)\big) \sim m_\alpha(\nu) \bar{\mu}(x).
	\]
\end{proposition}	

As a conclusion from our main results we obtain an analogue of the Breiman Lemma for free multiplicative convolution. This theorem answers an open question posed in \cite[Section 2.4 (1)]{CH18} about the behavior of the tail of $\mu\boxtimes\nu$. 

\begin{lemma}\label{lem:Breiman}
	Let $L\in \mathcal{R}_0$.	If $\mu,\nu\in\mathcal{M}_+$ are such that
	\[
\mu\big((x,+\infty))\sim x^{-\alpha}L(x)\qquad\mbox{ and }\qquad m_{\lfloor \alpha+1\rfloor}(\nu)<+\infty,
	\]
	for $\alpha\geq0$,  then
	\[
	(\mu\boxtimes \nu)\big((x,+\infty)\big)\sim m_1^\alpha(\nu) \bar{\mu}(x).
	\]
\end{lemma}

\begin{proof}
	We have $\mu\in\mathcal{M}_p$ for some $p\in\mathbb{N}_0$.
	We consider four cases:
	\begin{enumerate}
		\item[(1)] $p=0$, $\alpha\in[0,1)$.
		 
		 We have
		\[
		S_{\mu\boxtimes \nu}\left(-\frac1x\right) = 		S_{\mu}\left(-\frac1x\right) 		S_{ \nu}\left(-\frac1x\right) \sim \frac{1}{m_1(\nu)} S_\mu\left(-\frac1x\right)
		\]
		and the result follows from Corollary \ref{cor:first0} for $\alpha\in(0,1)$ and by Corollary \ref{cor:firstA0} for $\alpha=0$.		
		\item[(2)] $p=0$, $\alpha=1$.
		
		If $\alpha=1$, then we require that $m_2(\nu)<+\infty$. Thus, $S_\nu(z)={1}/{m_1(\mu)}+o(z)$. This implies that
		\[
		S_{\mu\boxtimes \nu}\left(-\frac1x\right)  = m_1(\nu)^{-1}S_\mu\left(-\frac1x\right)+o(x^{-1})
		\]
		and Corollary \ref{cor:firstA1} gives the assertion.
		\item[(3)] $p\in \mathbb{N}$, $\alpha\in[p,p+1)$.
		
		\[
		S_{\mu\boxtimes \nu}^{(p)}\left(-\frac1x\right) \sim		S_{\mu}^{(p)}\left(-\frac1x\right) 		S_{ \nu}\left(-\frac1x\right) \sim \frac{1}{m_1(\nu)} S_\mu^{(p)}\left(-\frac1x\right) 	
		\]
		and the result follows from Corollary \ref{cor:first}.
		\item[(4)] $p\in\mathbb{N}$, $\alpha=p+1$.
		
		 In view of Lemma \ref{lem:expansion}, we have 
		\[
		S_{\mu\boxtimes \nu}^{(p)}\left(-\frac1x\right) =		S_{\mu}^{(p)}\left(-\frac1x\right) 		S_{ \nu}\left(-\frac1x\right) +c+o(x^{-\eps})= \frac{1}{m_1(\nu)} S_\mu^{(p)}\left(-\frac1x\right) +c+o(x^{-\eps})
		\]
		and the result follows from Corollary \ref{cor:first}.
	\end{enumerate}
\end{proof}

Lemma \ref{lem:Breiman} was already observed in \cite[Corollary 4.1]{CCH18} in the special case when $\nu$ is the Marchenko-Pastur law with rate $1$ and scale parameter $1$ for which we have $m_1(\nu)=1$.

A somehow similar results were obtained in \cite{CH18} for Boolean multiplicative and additive convolutions. 

\subsection{Tails at $0^+$}\label{subsec:zerotail}
In this subsection we observe that the results about right tail can be immediately applied to determine the behavior of the tail at $0^+$. This follows from the fact that $S$-transform behaves in a tractable way when one considers the pushforward measure under the mapping $x\mapsto 1/x$.

\begin{remark}\label{rem:inv}
	Let $\hat{\mu}$ denote the pushforward measure of $\mu$ by the mapping $x\mapsto x^{-1}$. If $\mu$ has a regularly varying tails of $0^+$ with index $-\alpha$, then $\hat{\mu}$ has a regularly varying tails at $\infty$ with the same index. Indeed, we have
	\[
	\hat{\mu}\big((x,+\infty)\big) = \mu\big([0,x^{-1})\big)\sim x^{-\alpha} L(x)
	\]
	and
	\[
	m_1(\hat{\mu}) = m_{-1}(\mu).
	\]
	Moreover, by Proposition \ref{thm:useful} $(v)$ we have
	\[
	S_{\hat{\mu}}(z)=1/S_{\mu}(-1-z),\qquad z\in(-1,0).
	\]
\end{remark}

\begin{defin}
	We say that $\mu$ has regularly varying tail at $0^+$ with index $-\alpha\leq0$ if
	\begin{align}\label{eq:L0}
		\mu\big([0,x^{-1})\big)\sim \frac{L(x)}{x^{\alpha}}.
	\end{align}
	for some slowly varying function $L$.
\end{defin}

Let us list here some immediate consequences of the result from Section \ref{sec:4}.
\begin{Corollary} Assume $L\in\mathcal{R}_0$.
	\begin{enumerate}[(i)]
		\item 	Let $\alpha=0$. If \eqref{eq:L0} holds, then 
		\begin{align}\label{eq:SinKR0}
			x\mapsto S_\mu(-1+1/x) \in K\mathcal{R}_\infty.
		\end{align}
		Conversely, if \eqref{eq:SinKR0} holds, then \eqref{eq:L0} holds with
		\[
		L(x)=1/g_{\mu}^{\langle-1\rangle}(x),
		\]
		where 
		$g_\mu(x):=-\chi_\mu(-1+1/x)=(x-1)S_\mu(-1+1/x)$.
		\item 	Let $\alpha\in(0,1)$. Then, \eqref{eq:L0} is equivalent to
		\begin{align*}
			S_{\mu}\left(-1+\frac1x\right)&\sim \left( \frac{\pi \alpha}{\sin(\pi \alpha)} \right)^{1/\alpha}x^{1/\alpha-1}M^\#(x^{1/\alpha}),
		\end{align*}
		where $M^\#$ is the de Bruijn conjugate of a slowly varying function $M(x)=L^{-1/\alpha}(x)$.
\item 	Let $\alpha=1$. Then, \eqref{eq:L0} is equivalent to
\begin{align*}
	x\mapsto S_\mu\left(-1+1/x\right)\in \Pi_M(1),
\end{align*}
where $M(x):=L\big((x-1) S_\mu(-1+1/x)\big)$.
\item 	Let $\alpha\in[p,p+1]$ for $p\in\mathbb{N}$. 
\begin{enumerate}[a)]
	\item If  one of the following conditions holds
	\begin{itemize}
		\item $\alpha\in(p,p+1)$, or
		\item $\alpha=p$ and $\int_1^{+\infty} L(t)/t\,\dd t<+\infty$,
	\end{itemize} 
then \eqref{eq:L0} is equivalent to 
\begin{align*}
	S_\mu^{(p)}\left(-1+\frac1x\right)\sim (-1)^{p}\frac{\Gamma(\alpha+1)\Gamma(p+1-\alpha)}{m_{-1}(\mu)^{\alpha-1}} x^{p+1-\alpha}L(x).
\end{align*}
\item If $\alpha=p+1$ and $\int_1^{+\infty} L(t)/t\,\dd t=+\infty$, then \eqref{eq:L0} is equivalent to 
\[
x\mapsto 	S_\mu^{(p)}\left(-1+\frac1x\right) \in \Pi_L\left((-1)^p\frac{(p+1)!}{m_{-1}(\mu)^{p}}\right).
\]
\end{enumerate} 
	\end{enumerate}
\end{Corollary}

We also note that recently in \cite[Remark 3.2]{Ji21} determined the tail at $0^+$ of $\mu\boxtimes\nu$ when both measures have regularly varying tail at $0^+$ with indexes in $(0,1)$. This falls under the case $(ii)$ of the above result.

Let us determine here left and right tails of a family of probability measures considered in \cite{HaM13}.
\begin{example}
	In \cite{HaM13} the authors considered family of probability measures $(\mu_{\alpha,\beta}\colon \alpha, \beta\geq0)$ defined by their $S$-transforms
	\[
	S_{\mu_{\alpha,\beta}}(z)=\frac{(-z)^\beta}{(1+z)^{\alpha}}.
	\]
	In particular, $\mu_{0,1 }$ is the image of the
	free Poisson distribution with shape parameter under the map $x\mapsto x^{-1}$ and has a density
	\[
	\mu_{0,1}(\dd x)=\frac{1}{2\pi}\frac{\sqrt{4x-1}}{x^2}\textbf{1}_{(1/4,+\infty)}(x)\dd x.
	\]
	This is also the free stable distribution with parameters $\alpha=1/2$ and $\rho=1$ (\cite[Appendix A1]{BPB}). In general case, the asymptotic behavior of tails $\bar{\mu}_{\alpha,\beta}$, if non trivial,  can be obtained form the implicit description of densities of $\mu_{\alpha,\beta}$ from \cite[Theorems 4,5,6]{HaM13} or by the use of
	standard Tauberian theorems applied to results of \cite[Theorem 3]{HaM13}. 
	Thanks to the results of Section \ref{sec:4} and this section we are able to deduce the precise asymptotics of $x\mapsto\bar{\mu}_{\alpha,\beta}(x)$ and $x\mapsto \mu\big([0,1/x)\big)$ just by examining the asymptotics of corresponding $S$-transform. 
	We have 
	\[
	S_{\mu_{\alpha,\beta}}\left(-\frac1x\right)=\frac{x^{\alpha-\beta }}{(x-1)^{\alpha}}\sim x^{-\beta}
	\]
	and
	\[
	S_{\mu_{\alpha,\beta}}\left(-1+\frac1x\right)=x^\alpha \left(1-\frac1x\right)^\beta\sim x^\alpha.
	\]
	If $\beta>0$, then we obtain
	\[
	\lim_{x\to+\infty} x^{1/(\beta+1)}\bar{\mu}_{\alpha,\beta}(x)=\frac{\sin(\pi \frac{1}{\beta+1})}{\pi \frac{1}{\beta+1}}
	\]
	If $\alpha>0$, then 
	\[
	\lim_{x\to+\infty} x^{1/(\alpha+1)}\mu_{\alpha,\beta}\big([0,x^{-1})\big)=\frac{\sin(\pi \frac{1}{\alpha+1})}{\pi \frac{1}{\alpha+1}}
	\]
\end{example}

\subsection{Symmetric measures}
The definition of the $S$-transform of $\mu$ when $m_1(\mu)=0$ is problematic, however in order to define $\mu\boxtimes\nu$ it is enough that one of the measures is supported on $[0,+\infty)$. In this subsection we apply our results in the case when $\mu$ is symmetric it is possible do to some observations from \cite{Octavio09}.
The notion of $S$-transform was first extended measures with $m_1(\mu)=0$ and  with all moments finite in \cite{Rao07}. In the case of symmetric measures it was then further generalized to arbitrary symmetric probability measures in \cite{Octavio09}. We say that a Borel measure $\mu$ is symmetric if $\mu(B)=\mu(-B)$ for all Borel sets $B$. Let $\mathcal{M}_S$ denote the class of symmetric Borel probability measures on $\R$. 

Define the moment transform $\psi_\mu$ of $\mu\in\mathcal{M}_S$ by 
\[
\psi_{\mu}(z)=\int_{\R} \frac{zt}{1-zt}\mu(\dd t),\qquad z\in i\R.
\]
Then, the function $\psi_\mu\colon i\R_-\to (\mu(\{0\})-1,0)$ is invertible and we denote its inverse by $\chi_\mu\colon (\mu(\{0\})-1,0)\to i\R_-$. The $S$-transform of $\mu\in\mathcal{M}_S$ is then defined by 
\[
S_\mu(z)=\frac{1+z}{z}\chi_\mu(z),\qquad z\in(\mu(\{0\})-1,0).
\]
Unlike the $S$-transform of measure in $\mathcal{M}_+$ (which was a real function), the $S$-transform of symmetric measures is imaginary. Eq. \eqref{eq:Strans} is still holds if one of the measures belongs to $\mathcal{M}_+$, \cite[Theorem 7]{Octavio09}. More precisely, if $\mu\in\mathcal{M}_S$, $\nu\in\mathcal{M}_+$ and both measures are not $\delta_0$, then $\mu\boxtimes\nu \in\mathcal{M}_S$ and 
\[
S_{\mu\boxtimes\nu}(z) = S_{\mu}(z)S_\nu(z),\qquad z\in(-\varepsilon,0)
\]
for some $\varepsilon>0$.

\begin{remark}
	In \cite{Octavio09} two $S$-transforms of a symmetric measure. The second one corresponds to the fact that function $\psi_\mu\colon i\R_+\to (\mu(\{0\})-1,0)$ is also invertible and we may take use it instead of $\chi_\mu$ to define the $S$-transform.
	
	Both $S$-transforms in \cite{Octavio09} were defined for complex argument, but for our purposes it is enough to consider real ones.
\end{remark}

Let $\mu^2\in\mathcal{M}_+$ denote the pushforward measure of $\mu\in\mathcal{M}_S$ by the mapping $x\mapsto x^2$. Clearly, 
\begin{align}\label{eq:mu2}
\mu\big((x,+\infty)\big)=\frac12 \mu^2\big((x^2,+\infty)\big).
\end{align}
Moreover, the $S$-transforms of $\mu\in\mathcal{M}_S$ and of $\mu^2\in\mathcal{M}_+$ are  related by the following equation (see \cite[Theorem 6 b)]{Octavio09}) for $\mu\neq \delta_0$ there exists $\varepsilon>0$ such that 
\begin{align}\label{eq:Ssym}
S_{\mu}(z)^2=\frac{1+z}{z}S_{\mu^2}(z),\qquad z\in(-\varepsilon,0).
\end{align}

Finally, it is easy to see that $x\mapsto \mu\big((x,+\infty)\big)$ is regularly varying with index $-\alpha\leq0$ if and only if $x\mapsto \mu^2\big((x,+\infty)\big)$ is regularly varying with index $-\alpha/2$. Thus, we may apply results of Section \ref{sec:4} to completely characterize the behavior of the $S$-transform of symmetric probability measures with regularly varying tail. We present an example below. 
\begin{example}
Let us consider $\mu\in\mathcal{M}_S$ for which 
\begin{align}\label{eq:musym1}
\lim_{z\to 0^-} S_\mu(z) = c\,i
\end{align}
for some $c>0$. Then \eqref{eq:Ssym} implies that 
\[
S_{\mu^2}\left(-\frac1x\right) \sim c^2\frac{1}{x}. 
\]
By the converse implication of Theorem \ref{thm:01} with $\alpha=1/2$ we obtain 
\[
\mu^2\big((x,+\infty)\big) \sim \frac{2}{\pi c} x^{-1/2}.
\]
Thus, by \eqref{eq:mu2} we finally arrive at
\begin{align}\label{eq:musym2}
\mu\big((x,+\infty)\big) \sim \frac{1}{\pi c} \frac1x.
\end{align}
Clearly, all above steps can be reverted and we therefore see that \eqref{eq:musym1} and \eqref{eq:musym2} are equivalent.
\end{example}

\section*{Acknowledgment} 
We thank Serban Belinschi for his suggestion to consider tails of $\boxtimes$-ID measures.

Research was funded by Warsaw University of Technology within the Excellence Initiative:
Research University (IDUB) programme.

\bibliographystyle{plain}

\bibliography{Bibl}

\end{document}